\documentclass{amsart}

\usepackage{amssymb}

\usepackage{graphicx}

\usepackage{mathenv,amsfonts, amsmath, amsthm, amssymb, epsfig,
mathrsfs, psfrag,color, subfigure}
\newcommand{\ranglemp}[1]{\rangle_{-\frac{1}{2},\frac{1}{2},#1}}
\newcommand{\ranglepp}[1]{\rangle_{\frac{1}{2},#1}}
\newcommand{\ranglemm}[1]{\rangle_{-\frac{1}{2},#1}}
\newcommand{\psie}[1]{\hat{\mathsf e}^{#1}}
\newcommand{\phie}[1]{\mathsf e^{#1}}

\newcommand{\SL}[1]{{\mathscr S}_{#1}}
\newcommand{\Hp}[1]{H^\frac{1}{2}(#1)}
\newcommand{\Hcp}[1]{\mathcal H^{\frac{1}{2}}(#1)}
\newcommand{\Hm}[1]{H^{-\frac{1}{2}}(#1)}
\newcommand{\Hcm}[1]{\mathcal H^{-\frac{1}{2}}(#1)}

\renewcommand{\Re}{{\rm Re\,}}
\renewcommand{\Im}{{\rm Im\,}}

\usepackage{lmodern}
\DeclareMathAlphabet{\mathbbo}{U}{bbold}{m}{n}
\newcommand{\1}[1]{\mathbbo{1}_{#1}}

\copyrightinfo{2009}{Brown University}

\newtheorem{theorem}{Theorem}[section]
\newtheorem{lemma}[theorem]{Lemma}

\theoremstyle{definition}
\newtheorem{definition}[theorem]{Definition}

\newtheorem{proposition}[theorem]{Proposition}
\newtheorem{conj}[theorem]{Conjecture}

\theoremstyle{remark}
\newtheorem{remark}[theorem]{Remark}

\numberwithin{equation}{section}

\begin{document}

\title[Calder\'on cavities inverse problem]{Calder\'on cavities inverse problem as a shape-from-moments problem}


\author{Alexandre Munnier}
\address{Universit\'e de Lorraine, CNRS, Inria, IECL, F-54000 Nancy, France}
\email{alexandre.munnier@univ-lorraine.fr}

\author{Karim Ramdani}
\address{Universit\'e de Lorraine, CNRS, Inria, IECL, F-54000 Nancy, France}
\email{karim.ramdani@inria.fr}


\date{}

\dedicatory{}

\begin{abstract}
In this paper, we address a particular case of Calder\'on's (or conductivity) inverse problem in dimension two, namely the case of a homogeneous background containing a finite number of cavities (i.e. heterogeneities of infinitely high conductivities). We aim to recover the location and the shape of the cavities from the knowledge of the Dirichlet-to-Neumann (DtN) map of the problem. The proposed reconstruction method is non iterative and uses two main  ingredients. First, we show how to compute the so-called generalized P\'olia-Szeg\"o tensors (GPST) of the cavities from the DtN of the cavities. 
Secondly, we show that the obtained shape from GPST inverse problem can be transformed into a shape from moments problem, for some 
particular configurations. However, numerical results suggest that the reconstruction method is efficient for arbitrary geometries.  
\end{abstract}

\maketitle


\section{Introduction}
Let be given a simply connected open bounded set $\Omega$ in $\mathbb R^2$ with Lipschitz  boundary $\Gamma$. Let $\sigma$ be a positive function in $L^\infty(\Omega)$ and consider the elliptic boundary value problem:
\begin{subequations}
\label{calderon}
\begin{alignat}{3}
-\nabla\cdot(\sigma\nabla u^f)&=0 &\quad&\text{in }\Omega\\
u^f&=f&&\text{on }\Gamma.
\end{alignat}
Calder\'on's inverse conductivity problem \cite{Cal80} is to recover the conductivity $\sigma$ knowing the Dirichlet-to-Neumann (DtN) map $f\longmapsto {\partial_nu^f}_{|\Gamma}$ of the problem.
\end{subequations}

\begin{figure}[h]
\centering
\includegraphics[width=.3\textwidth]{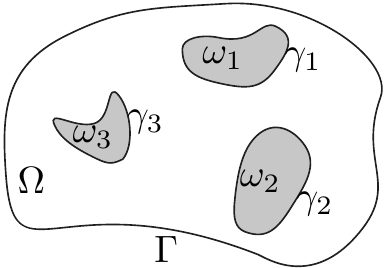}
\caption{The multiply connected cavity.}
\label{fig:cavity}
\end{figure}
In a recent work \cite{MunRam17}, the authors investigated this problem in the particular case of piecewise conductivity with infinitely high contrast (see for instance Friedman and Vogelius \cite{FriVog89} who considered this problem in the case of small inclusions). Combining an integral formulation of the problem with tools from complex analysis, they proposed an explicit reconstruction formula for the geometry of the unknown cavity. However, due to the crucial use of the Riemann mapping theorem, the proposed approach was limited to the case of a single cavity. The aim of this paper is to investigate the case of a multiply connected cavity. More precisely, we suppose that $\Omega$ contains a multiply connected domain
$\omega=\cup_{k=1}^N\omega_k$, where the open sets $\omega_{k}$, for $k=1,\dots,N$ are non intersecting simply connected domains with $C^{1,1}$  boundaries $\gamma_{k}$ and $\overline{\omega}\subset\Omega$ (see Figure~\ref{fig:cavity}).
We denote by  $\gamma=\cup_{k=1}^N \gamma_{k}$ and by $n$ the unit normal to
$\Gamma\cup \gamma$ directed towards the exterior of $\Omega\setminus\overline{\omega}$.  

For every $f$ in $H^{\frac{1}{2}}(\Gamma)$, let  $(u^f,{\mathbf c}^f)\in
H^1(\Omega\setminus\overline{\omega})\times\mathbb R^N$, with ${\mathbf c}^f:=(c^f_{1}, \dots, c^f_{N})^{\mathbf T}$, be the solution of the Dirichlet
problem:
\begin{subequations}
\label{main_problem}
\begin{alignat}{3}
-\Delta u^f&=0&\quad&\text{in }\Omega\setminus\overline{\omega}\\
u^f&=f&&\text{on }\Gamma\\
u^f&=c^f_{k}&&\text{on }\gamma_{k}, \qquad k=1, \dots,N,
\end{alignat}
with the additional circulation free conditions:
\begin{equation}
\label{free_circ}
\int_{\gamma_{k}} \partial_n u^f\,{\rm d}\sigma =0,  \qquad k=1, \dots,N.
\end{equation}
\end{subequations}
By following the proof given in the Appendix of \cite{MunRam17} for the case of a single cavity ($N=1$), it can be easily shown that this elliptic problem is well-posed and that its solution can be seen as the limit solution obtained by considering problem \eqref{calderon} for a piecewise constant conductivity and letting the constant conductivity inside the cavities tend to infinity (at the same speed). 

The inverse problem investigated in this paper can be formally stated as follows (the exact functional framework will be made precise later on): \emph{knowing the Dirichlet-to-Neumann (DtN) map $\Lambda_\gamma:f\longmapsto
{\partial_nu^f}_{|\Gamma}$, how to reconstruct the multiply connected cavity $\omega$ ?}\\
Roughly speaking, one can distinguish in the literature two classes of approaches for shape identification: iterative and non iterative methods (see for instance the survey paper by Potthast \cite{Pot06}). In the first class of methods, one computes a sequence of approximating shapes, generally by solving at each step the direct problem and using minimal data. Among these approaches, we can mention those based on optimization \cite{Bor02,CapFehGou09}, on the reciprocity gap principle \cite{KreRun05,IvaKre06,CakKre13}, on the quasi-reversibility \cite{BouDar10,BouDar14} or on conformal mapping \cite{AkdKre02,Kre04,HadKre05,HadKre06,HadKre10,Kre12,HadKre14}. The second class of methods covers non iterative methods, generally based on the construction of an indicator function of the inclusion(s). These sampling/probe methods do not need to solve the forward problem, but require the knowledge of the full DtN map. With no claim as to completeness, let us mention the enclosure and probe method of Ikehata \cite{Ike98,Ike00,Ike00b,IkeSil00,ErhPot06}, the linear sampling method \cite{ColKir96,ColPiaPot97,CakCol14}, Kirsch's Factorization method \cite{BruHan00,HanBru03,Kir05} and Generalized Polya-Szeg\"o Tensors 
\cite{AmmKan04,AmmKan06,AmmKan07,AmmGarKan14, MMA:MMA3195}.\\
The reconstruction method proposed in this paper is non iterative and can be decomposed into two main steps. First, we show that the knowledge of DtN map gives access to the so-called Generalized P\'olya-Szeg\"o Tensors (GPST) of the cavity. This is done (see \S.~\ref{subsec:DtNGPST}) by adapting to the multiply connected case the boundary integral approach proposed in \cite{MunRam17} for a simply connected cavity. The second step is to transform this shape from GPST problem into a (non standard) shape from moments problems (see \S.~\ref{subsec:moment}). Reconstructing the geometry of the cavities amounts then to reconstructing the support of a density from the knowledge of its harmonic moments. Our reconstruction algorithm is then obtained by seeking a finite atomic representation of the unknown measure. Let us emphasize that we have been able to justify the connection between the  shape from GPST problem and the shape from moments problem only in some particular cases (for a single cavity, for two disks and for small cavities). However, the reconstruction method turns out to be numerically efficient for arbitrary cavities. \\
The paper is organized as follows. We collect some technical material from potential theory in Section \ref{sect:EI}. The reconstruction method is described in \ref{sect:rec}. Section \ref{sect:proof} is devoted to the proof of Theorem \ref{thm:conj}. Finally, examples of numerical reconstructions are given in Section \ref{sect:num}. 

\section{Background on potential theory}
\label{sect:EI}
This section aims to revisit the results from potential theory given in \cite[Section 2.1.]{MunRam17} in the context of a multiply connected cavity. For the proofs, we refer the reader to \cite{MunRam17} and to the books of McLean \cite{McL00}, Steinbach \cite{Ste08b} or Hsiao and Wendland \cite{HsiWen08} for more classical material. Denote by
$$
G(x)= -\frac{1}{2\pi}\log|x|
$$
the fundamental solution of the operator $-\Delta$ in $\mathbb R^2$. We pay careful attention to state the results in a form that includes multiply connected boundaries.
\subsection{Single layer potential}
We define the function spaces 
$$\Hcp{\gamma}:=H^{\frac 12}(\gamma_{1})\times H^{\frac 12}(\gamma_{2})\times\dots\times H^{\frac 12}(\gamma_{N}),$$
$$
\Hcm{\gamma}:=H^{-\frac 12}(\gamma_{1})\times H^{-\frac 12}(\gamma_{2})\times\dots\times H^{-\frac 12}(\gamma_{N}),
$$
which are Hilbert spaces when respectively endowed with the norms
$$ \| q\|_{\frac{1}{2},\gamma} = \left(\| q_{1}\|_{\frac{1}{2},\gamma_{1}}^2+\dots+\| q_{N}\|_{\frac{1}{2},\gamma_{N}}^2\right)^{\frac 12},\qquad \forall\, q=( q_{1},\dots, q_{N})\in  \Hcp{\gamma},$$
$$ \|\hat q\|_{-\frac{1}{2},\gamma} = \left(\|\hat q_{1}\|_{-\frac{1}{2},\gamma_{1}}^2+\dots+\|\hat q_{N}\|_{-\frac{1}{2},\gamma_{N}}^2\right)^{\frac 12},\qquad \forall\,\hat q=(\hat q_{1},\dots,\hat q_{N})\in  \Hcm{\gamma}.$$
\begin{definition}
For every $\hat q=(\hat q_{1},\dots,\hat q_{N})\in \mathcal H^{-\frac{1}{2}}(\gamma)$, we denote by $\mathscr S_{\gamma} \hat q$ the single layer potential associated with the density $\hat q$.\end{definition}
Given $\hat q=(\hat q_{1},\dots,\hat q_{N})\in \mathcal H^{-\frac{1}{2}}(\gamma)$, it is well-known that the single layer potential $\mathscr S_{\gamma} \hat q$ defines a harmonic function in $\mathbb R^2\setminus \gamma$. Furthermore, if $\hat q_{k}\in L^2(\gamma_{k})$ for all $k\in\{1,\dots,N\}$, we can write:
$$\mathscr S_{\gamma} \hat q(x)=\int_{\gamma} G(x-y) \hat q(y)\,{\rm d}\sigma_y
=\sum_{k=1}^N \mathscr S_{\gamma_{k}} \hat q_{k}(x), \qquad \forall\,x\in\mathbb R^2\setminus\gamma.$$
The single layer potential defines a bounded linear operator from $\mathcal H^{-\frac{1}{2}}(\gamma)$ into $H^1_{\rm \ell oc}(\mathbb R^2)$, and the asymptotic behavior of $\mathscr S_{\gamma} \hat q$ reads as follows (see for instance \cite[p.~261]{McL00})
\begin{equation}\label{asymptSC}
\mathscr S_{\gamma} \hat q(x) = -\frac{1}{2\pi}\langle \hat q, 1\ranglemp{{\gamma}}\, \log |x| + O(|x|^{-1}),
\end{equation}
where we have set for every $\hat q\in \Hcm{\gamma}$ and $p\in \Hcp{\gamma}$:
$$\langle \hat q, p\ranglemp{{\gamma}}  := \sum_{k=1}^N \langle \hat q_{k} , p_{k}\ranglemp{{\gamma_{k}}},$$
in which $ \langle \cdot, \cdot\ranglemp{{\gamma_{k}}}$ stands for the duality brackets between $\Hm{\gamma_{k}}$ and $\Hp{\gamma_{k}}$. This shows in particular that $\mathscr S_{\gamma} \hat q$ has finite Dirichlet energy (i.e. $\nabla(\mathscr S_{\gamma} \hat q)\in (L^2(\mathbb R^2)^2$) for all $\hat q=(\hat q_{1},\dots,\hat q_{N})$ in the function space 
$$\widehat  {\mathcal H} (\gamma):=\widehat H(\gamma_{1})\times \widehat H(\gamma_{2})\times\ldots\times \widehat H(\gamma_{N}),$$
where for every $k\in\{1,\ldots,N\}$:
$$\widehat H(\gamma_{k}):=\{\hat q_{k}\in \Hm{\gamma_k}\,:\,\langle \hat q_{k},1\ranglemp{\gamma_k}=0\}.$$
\begin{remark}
It is worth noticing that condition $\hat q\in \widehat  {\mathcal H} (\gamma)$ is only a sufficient condition ensuring $\nabla (\mathscr S_{\gamma} \hat q)\in (L^2(\mathbb R^2))^2$, since a necessary and sufficient condition is: $\langle \hat q, 1\ranglemp{{\gamma}}=0$. However, considering conditions \eqref{free_circ}, $\widehat  {\mathcal H} (\gamma)$ is clearly the appropriate function space to tackle the cavity problem \eqref{main_problem}.
\end{remark}
We also recall that the single layer potential satisfies the following classical jump conditions 
\begin{equation}\label{eq:jump}
[\mathscr S_{\gamma} \hat q\, ]_{\gamma_{k}}=0, \qquad \qquad 
\left[ \partial_{n}(\mathscr S_{\gamma} \hat q) \right]_{\gamma_{k}}=-\hat q_{k}, \qquad  \forall\,\hat q=(\hat q_{1},\dots,\hat q_{N})\in \Hcm{\gamma}.
\end{equation}
In the above relations, we have used the notation $[u]_{\gamma_{k}}=u_i|_{\gamma_{k}}-u_e|_{\gamma_{k}}$ and $[\partial_{n}u]_{\gamma_{k}}=(\partial_{n}u_i)|_{\gamma_{k}}-(\partial_{n}u_e)|_{\gamma_{k}}$, where $u_i$ and $u_e$ denote respectively the restrictions of a given function $u$ to the interior and exterior of $\gamma_{k}$. Let us emphasize that these classical formulae (and other trace formulae detailed below) may be impacted by different conventions concerning the definition of the fundamental solution, the unit normal or the jumps.\\
Let us focus now on the trace of the single layer potential.
\begin{definition}
For every $\hat q=(\hat q_{1},\dots,\hat q_{N})\in \Hcm{\gamma}$, we denote by $q={\mathsf S}_{\gamma}\hat q\in \mathcal H^{\frac12}(\gamma)$ the trace  on $\gamma$ of the  single layer potential ${\mathscr S}_{\gamma}\hat q$:
$$
q=(q_{1},\dots,q_{N}), \qquad q_{i}= \sum_{j=1}^{N}\mathsf S_{ij}\hat q_j,  
$$
where 
$$
\mathsf S_{ij}\hat q_j={\rm Tr}_{\gamma_i}(\mathscr S_{\gamma_j}\hat q_j),\qquad \forall\,i,j\in\{1,\ldots,N\}.$$
\end{definition}
Note that $\mathsf S_\gamma=\left(\mathsf S_{ij}\right)_{1\leqslant i,j\leqslant N}:\mathcal H^{-\frac 12}(\gamma)\to\mathcal H^{\frac12}(\gamma)$ defines a bounded linear operator with weakly singular kernel. Hence, if $\hat q=(\hat q_{1},\dots,\hat q_{N})\in \Hcm{\gamma}$ is such that $\hat q_{k}\in L^\infty(\gamma)$ for all $k\in\{1,\dots,N\}$, then
$${\mathsf S}_{\gamma} \hat q(x)=\int_{\gamma} G(x-y) \hat q(y)\,{\rm d}\sigma_y, \quad\forall\, x\in \gamma.$$
Using Green's formula and the asymptotics \eqref{asymptSC}, we can easily prove the identity
\begin{equation}
\label{equiv:decroi}
\langle \hat q,\mathsf S_\gamma \hat q  \ranglemp{{\gamma}} = \int_{\mathbb R^2} \left|\nabla (\mathscr S_{\gamma} \hat q)\right|^2{\rm d}x<+\infty, \qquad \forall\,\hat q\in {\widehat {\mathcal H}}({\gamma}).
\end{equation}
According to \cite[Theorem 8.12]{McL00}, $\mathsf S_\gamma:\mathcal H^{-\frac 12}(\gamma)\to\mathcal H^{\frac12}(\gamma)$ defines a strictly positive-definite operator on ${\widehat {\mathcal H}}$ (since ${\widehat {\mathcal H}}\subset\{\hat q \in \Hcm{\gamma} : \langle \hat q, 1\ranglemp{{\gamma}}=0\} $). It is also known (see \cite[Theorem 8.16]{McL00}) that $\mathsf S_\gamma$ is boundedly invertible if and only if the logarithmic capacity of $\gamma$ (see \cite[p.~264]{McL00} for the definition) satisfies ${\rm Cap}({\gamma})\neq 1$. From now on, and without loss of generality, let us assume that  the diameter of $\Omega$ is less than 1 (otherwise, it suffices to rescale the problem), which implies in particular that ${\rm Cap}(\Gamma)< 1$ and ${\rm Cap}(\gamma)< 1$ (see \cite[p.~143]{Ste08b} and references therein). 

In order to characterize the image of $\hat {\mathcal H}(\gamma)$ by $ \mathsf S_\gamma$, we need to introduce the following densities. 
\begin{definition}
\label{def:psieq}
For every $k\in\{1,\ldots,N\}$, we define the unique density:
$$\psie{k}:=(\psie{k}_1,\ldots,\psie{k}_N)\in \Hcm{\gamma}$$ 
such that $\mathsf S_\gamma \psie{k}$ is constant on $\gamma_\ell$ for every $\ell \in\{1,\ldots,N\}$ and satisfying the circulation conditions: 
$$\langle \psie{k}_\ell,1\ranglemp{\gamma_\ell}=\delta_{k,\ell},\qquad \ell =1, \dots, N$$
where $\delta_{k,\ell}$ denotes Kronecker's symbol.
\end{definition}
The existence and uniqueness of such functions $\psie{k}$ is ensured by Lemma \ref{lem:EImult} of the Appendix (simply take $f=0$ and $\mathbf b$ as the $k-$th element of the canonical basis of $\mathbb R^N$). Furthermore, the family $\{\psie{k}, k=1,\dots, N\}$ is obviously linearly independent in $\Hcm{\gamma}$ and thus, so is the family $\{\phie{k}:= \mathsf S \psie{k}, k=1,\dots, N\}$ in $\Hcp{\gamma}$.\\
\begin{proposition}
The operator $\mathsf S_\gamma$ defines an isomorphism from $\hat {\mathcal H}(\gamma)$ onto 
$${\mathcal H}(\gamma):=\{q\in \Hcp{\gamma} \,:\,\langle \psie{k},   q\ranglemp{\gamma}=0, \ k=1,\dots,N\}.$$
\end{proposition}
\begin{proof}
We only need to prove that ${\mathcal H}(\gamma)=\mathsf S_\gamma(\hat {\mathcal H}(\gamma))$. Let $\hat q\in \Hcm{\gamma}$ and set $q:=\mathsf S_\gamma \hat q$. We note that for all $i\in\{1,\dots,N\}$:
$$\langle \psie{i},q \ranglemp{\gamma}
=\langle \hat q,\phie{i}   \ranglemp{\gamma}
=\sum_{j=1}^N \langle \hat q_{j}, \phie{i}_{j} \ranglemp{{\gamma_{j}}}
=\sum_{j=1}^N \phie{i}_{j}  \ \langle \hat q_{j}, 1\ranglemp{{\gamma_{j}}},$$
where the last equality follows from the fact that $\phie{i}_{j}$ is constant on each boundary $\gamma_{j}$. The matrix $(\phie{i}_{j})_{1\leqslant i,j\leqslant N}$ being invertible, we have $\hat q\in \hat {\mathcal H}(\gamma)$ if and only if $q\in {\mathcal H}(\gamma)$. 
\end{proof}
The above result allows us to use the linear operator:
$$ 
{\mathsf S}_{\gamma}:  \hat q\in {\widehat {\mathcal H}}({\gamma})\longmapsto q\in
{\mathcal H}({\gamma}),$$
to identify any density $\hat q \in  {\widehat {\mathcal H}}({\gamma})$ with the trace 
$$q:={\mathsf S}_{\gamma}  \hat q \in {\mathcal H}({\gamma}).$$
Throughout the paper, we will systematically use this identification, using the notation with (respectively without) a hat on single layer densities of $ {\widehat {\mathcal H}}({\gamma})$ (respectively traces of single layer densities). 
\begin{definition}
\label{inner} 
For all $\hat p,\hat q\in {\widehat {\mathcal H}}({\gamma})$, we set:
	$$\langle \hat p,\hat q\ranglemm{\gamma}=\langle p,q\ranglepp{\gamma}=\langle \hat p,q \ranglemp{\gamma}.$$
\end{definition}
Obviously, using these inner products, the isomorphism ${\mathsf S}_{\gamma}$ turns out to be an isometry between the  spaces ${\widehat {\mathcal H}}({\gamma})$ and ${\mathcal H}({\gamma})$:
$$\|\hat q\|^2_{-\frac{1}{2},\gamma}=\|q\|^2_{\frac{1}{2},\gamma}=\int_{\mathbb R^2}|\nabla(\mathscr S_{\gamma} \hat q)|^2\,{\rm d}x,\qquad \forall\,\hat q\in {\widehat {\mathcal H}}({\gamma}).
$$
The following orthogonal projections will be needed in the sequel.
\begin{definition}
\label{def:Pi}
Let $\Pi_{\gamma}$ and $\widehat\Pi_{\gamma}$ denote respectively the orthogonal projections  from ${\mathcal H}^{\frac{1}{2}}({\gamma})$ into ${\mathcal H}({\gamma})$ and from ${\mathcal H}^{-\frac{1}{2}}({\gamma})$ into ${\widehat {\mathcal H}}({\gamma})$.  
\end{definition}
It can be easily checked that:
\begin{alignat*}{3}
\forall\,  \hat q\in {\mathcal H}^{-\frac{1}{2}}({\gamma}):\qquad \widehat\Pi_{\gamma}\hat q &= \hat q-\sum_{k=1}^{N} \langle \hat q_{k},1\ranglemp{\gamma_{k}} \,\hat{\mathsf e}^{k}\\
\forall\,  q\in {\mathcal H}^{\frac{1}{2}}({\gamma}):\qquad (\Pi_{\gamma} q)_{k} &=  q_{k}- \langle \hat {\mathsf e}^{k},q\ranglemp{\gamma}  \, 1_{|\gamma_{k}}, \qquad k=1,\dots, N.
\end{alignat*}
\begin{definition}\label{def:Tr0}
We denote by ${\mathsf{Tr}}_{\gamma}$  the classical trace operator (valued into ${\mathcal H}^{\frac{1}{2}}({\gamma})$), and by ${\mathsf{Tr}}^0_{\gamma}$ when it is left-composed with the orthogonal projection onto ${\mathcal H}({\gamma})$: ${\mathsf{Tr}}^0_{\gamma}:= \Pi_\gamma {\mathsf{Tr}}_{\gamma}$.
\end{definition}
Let us recall a useful characterization of the norm chosen on ${\mathcal H}(\gamma)$ (the proofs of the assertions stated below are given in \cite[Section 2.1.]{MunRam17} for the case of a simply connected cavity, but they can be easily extended to the multiply connected case).  
We define the quotient weighted Sobolev space:
$$W^1_0(\mathbb R^2)=\{u\in\mathcal D'(\mathbb R^2)\,:\,\rho u \in L^2(\mathbb R^2),\,\nabla u\in (L^2(\mathbb R^2))^2\}/\mathbb R,$$
where the weight is given by 
$$\rho(x):=\left(\sqrt{1+|x|^2}\log(2+|x|^2)\right)^{-1},\qquad x\in\mathbb R^2,$$
and where the quotient means that functions of $W^1_0(\mathbb R^2)$ are defined up to an additive constant. This space is a Hilbert space once equipped with the inner product:
$$\langle u,v\rangle_{W^1_0(\mathbb R^2)}:=\int_{\mathbb R^2}\nabla u\cdot\nabla v\,{\rm d}x.$$
For  $q\in \mathcal H(\gamma)$, and according to \eqref{equiv:decroi}, we have ${\mathscr S}_{\gamma}\hat q\in W^1_0(\mathbb R^2)$ and
$$
\|\hat q\|_{-\frac{1}{2},\gamma} = \|q\|_{\frac{1}{2},\gamma} = \|{\mathscr S}_{\gamma}\hat q\|_{W^1_0(\mathbb R^2)}, \qquad \forall\,q\in \mathcal H(\gamma).
$$
To conclude this subsection, let us recall that the normal derivative of the single layer potential is not continuous across $\gamma$ (see also \eqref{eq:jump}), as with obvious matrix notation (here the signs $+$ and $-$ refer respectively to the trace taken from the exterior and the interior of $\gamma$):
$$
\left(\partial_{n}(\mathscr S_{\gamma} \hat q)\right)_{|\gamma^\pm} = \left(\pm \frac12 + \mathsf L_{\gamma}^*\right)\hat q,
$$
where 
\begin{equation}\label{eq:Lstar1}
\mathsf L_\gamma^*=\left(\mathsf L_{ij}^*\right)_{1\leqslant i,j\leqslant N}:\mathcal H^{-\frac 12}(\gamma)\to\mathcal H^{-\frac12}(\gamma)
\end{equation}
where for smooth densities
\begin{equation}\label{eq:Lstar2}
(\mathsf L_{ij}^* \hat q_j)(x) = \int_{\gamma_j} \partial_{n_x}G(x-y)  q_{j}(y)\,{\rm d}\sigma_y, \qquad x\in \gamma_{i}.
\end{equation}

\subsection{Double layer potential}
For more details about the results recalled in this section, we refer the interested reader to the monographs by Hackbusch \cite[Chapter 7]{Hac95}, Kress \cite[Chapter 7]{Kre99} or Wen \cite[Chapter 4]{Wen92}.
\begin{definition}
For every $q=(q_{1},\dots, q_{N})\in \mathcal H^{\frac{1}{2}}(\gamma)$, we denote by $\mathscr
D_{\gamma}  q$ the double layer potential associated with the trace $q$. 
\end{definition}
Given $ q=( q_{1},\dots, q_{N})\in \mathcal H^{\frac{1}{2}}(\gamma)$, it is well-known that the double layer potential
$$\mathscr D_{\gamma}  q(x)=\int_{\gamma} \partial_{n_y}G(x-y)  q(y)\,{\rm d}\sigma_y
=\sum_{k=1}^N \mathscr D_{\gamma_{k}}  q_{k}(x), \qquad \forall\,x\in\mathbb R^2\setminus\gamma$$
defines a harmonic function in $\mathbb R^2\setminus \gamma$ whose normal derivative across $\gamma$ is continuous, but whose trace is not continuous. More precisely, we have
$$
{\rm Tr}_{\gamma^\pm}(\mathscr D_{\gamma} q) = \left(\mp \frac12 + \mathsf L_{\gamma}\right)q,
$$
where $\mathsf L_\gamma=\left(\mathsf L_{ij}\right)_{1\leqslant i,j\leqslant N}:\mathcal H^{\frac 12}(\gamma)\to\mathcal H^{\frac12}(\gamma)$, the adjoint of the operator defined in \eqref{eq:Lstar1}-\eqref{eq:Lstar2}, is given by 
$$
(\mathsf L_{ij} q_j)(x) = \int_{\gamma_j} \partial_{n_y}G(x-y)  q_{j}(y)\,{\rm d}\sigma_y, \qquad x\in \gamma_{i}.
$$
One has in particular
\begin{equation}\label{eq:jump2}
[\mathscr D_{\gamma}  q\, ]_{\gamma_{k}}=q.
\end{equation}

Later (see the proof of Theorem \ref{thm:quadrature2disks}), we will need to compute the inner product of two densities $p,q\in \mathcal H$ using double layer potentials. This is provided by the two next lemmas, which deal respectively with the cases of real and complex valued densities. 
\begin{lemma}\label{lem:dc1}
Given $p,q\in \mathcal H({\gamma})$, let $u:= {\mathcal S}_{\gamma}\hat p$ and $
v:= {\mathcal S}_{\gamma}\hat q$ denote the single layers respectively associated with $\hat p,\hat q\in {\widehat {\mathcal H}}({\gamma})$. Let $\widetilde u$ be the function defined in $\mathbb R^2\setminus\gamma$ by
$$
\widetilde u=
\begin{cases}
\widetilde u^-&\text{in }$\omega$\\
\widetilde u^+&\text{in }$\omega^c:=\mathbb R^2\setminus\overline{\omega}$.
\end{cases}
$$
where $\widetilde u^-$ (respectively $\widetilde u^+$) denotes a harmonic conjugate of $u$ in $\omega$ (respectively in $\omega^c$). Similarly, we define the function $\widetilde v$ associated to $v$. Then, the functions $u$ and $v$ admit double layer representation formulae associated to two densities $\check p,\check q\in   {\mathcal H}({\gamma})$: 
$$\widetilde u:= {\mathcal D}_{\gamma}\check p, \qquad\qquad
\widetilde v:= {\mathcal D}_{\gamma}\check q.$$
Moreover, we have
\begin{equation}\label{eq:pqcheck}
\langle p,q\ranglepp{\gamma}=-\langle \partial_{n}\widetilde u,\check q\ranglemp{\gamma}.
\end{equation}
\end{lemma}
\begin{proof} 
The harmonic conjugate $\widetilde u^-$ of $u$ in $\omega$ (i.e. the harmonic function such that $u+i \widetilde u^-$ is holomorphic in $\omega$) exists and is uniquely defined up to a constant (and this does not affect \eqref{eq:pqcheck}). The existence of the harmonic conjugate $\widetilde u^+$ of $u$ in $\omega^c$ is ensured by the fact that $u={\mathcal S}_{\gamma}\hat p$ is circulation free on $\gamma$ (since $p\in \widehat  {\mathcal H} (\gamma)$).
Moreover, thanks to Cauchy-Riemann's equations we have 
\begin{equation}\label{eq:gradperp}
\nabla u = (\nabla \widetilde u)^\perp \qquad \text{in }  \omega\cup \omega^c.
\end{equation}
Since $u$ admits by assumption a single layer  representation, it is continuous across $\gamma$, and hence, the same holds for its tangential derivative. According to \eqref{eq:gradperp}, this yields the continuity of $\partial_n\widetilde u$ across $\gamma$ (note that trace of $\widetilde u$ is not continuous across $\gamma$). From classical integral representation formula for harmonic functions, this implies that $\widetilde u$ admits a double layer representation: $\widetilde u:= {\mathcal D}_{\gamma}\check p$, with $\check p=[\widetilde u]_{\gamma}$. Defining similarly $\widetilde v:= {\mathcal D}_{\gamma}\check q$, we obtain using \eqref{eq:gradperp} that
$$
\langle p,q\ranglepp{\gamma} = \int_{\mathbb R^2}\nabla u\cdot \nabla v = \int_{\omega}\nabla \widetilde u\cdot \nabla \widetilde v +\int_{\omega^c}\nabla \widetilde u\cdot \nabla \widetilde v =-\langle \partial_{n}\widetilde u,[\widetilde v]\ranglemp{\gamma}, 
$$
where the last equality follows from Green's formula. Equation \eqref{eq:pqcheck} can then be deduced from the fact that $[\widetilde v]=[{\mathcal D}_{\gamma}\check q]=\check q$.
\end{proof} 
In the rest of the paper, we still denote by $\langle \cdot, \cdot \rangle_{\frac 12, \gamma}$ the hermitian inner product on $\mathcal H^{\frac12}(\gamma) $ seen as a complex Hilbert space. For instance, for $p=p^{1}+ip^{2}$ and $q=q^{1}+iq^{2}$ we will have
\begin{equation}\label{eq:psc}
\langle p, q \rangle_{\frac 12, \gamma} = \langle p^{1}, q^{1} \rangle_{\frac 12, \gamma} + \langle p^{2}, q^{2} \rangle_{\frac 12, \gamma} +i \left(\langle p^{2}, q^{1}\rangle_{\frac 12, \gamma}-\langle p^{1}, q^{2} \rangle_{\frac 12, \gamma}\right).
\end{equation}
Lemma \ref{lem:dc1} admits then the following counterpart for complex-valued densities. 
\begin{lemma}\label{lem:dc2} 
Given $p=p^{1}+ip^{2}$ and $q=q^{1}+iq^{2}$ with $p^{j},q^{j}\in \mathcal H$ ($j=1,2$), let us set
$$u:= {\mathcal S}_{\gamma}\hat p= u^{1}+iu^{2}, \qquad \qquad v:= {\mathcal S}_{\gamma}\hat q=v^{1}+iv^{2}$$ 
where $ u^{j}= {\mathcal S}_{\gamma}\hat p^{j}$ and $v^{j}:= {\mathcal S}_{\gamma}\hat q^{j}$. Define on $\mathbb R^2\setminus\gamma$ the complex-valued functions  $\widetilde u=\widetilde u^{1}+i\widetilde u^{2}$ and $\widetilde v=\widetilde v^{1}+i\widetilde v^{2}$ where $\widetilde u^j$ and $\widetilde v^j$ ($j=1,2$) are respectively the harmonic conjugates of $u^j$ and $v^j$ (as defined in Lemma \ref{lem:dc1}). Then, the functions $u$ and $v$ admit double layer representation formulae associated to two complex-valued densities $\check p,\check q\in   {\mathcal H}({\gamma})$: 
$$\widetilde u:= {\mathcal D}_{\gamma}\check p, \qquad\qquad
\widetilde v:= {\mathcal D}_{\gamma}\check q.$$
Moreover, we have
\begin{equation}\label{eq:pqcheckc}
\langle p,q\ranglepp{\gamma}=-\langle \partial_{n}\widetilde u,\check q\ranglemp{\gamma},
\end{equation}
where $\partial_{n}\widetilde u= \partial_{n}\widetilde u^{1}+i \partial_{n}\widetilde u^{2}$.
\end{lemma}
\begin{proof}
From \eqref{eq:pqcheck} and \eqref{eq:psc}, we immediately get that 
$$
\langle p, q \rangle_{\frac 12, \gamma} = - \langle \partial_{n}\widetilde u^{1}, \check q^{1} \ranglemp{\gamma} - \langle \partial_{n}\widetilde u^{2}, \check q^{2} \ranglemp{\gamma} -i \left(\langle \partial_{n}\widetilde u^{2}, \check q^{1}\ranglemp{\gamma}-\langle \partial_{n}\widetilde u^{1}, \check q^{2} \ranglemp{\gamma}\right)
$$
which is the claimed result.
\end{proof}
Let us conclude by recalling the relation (in dimension two) between the double layer potential and the (complex) Cauchy transform. In the sequel, we identify a point $x=(x_{1},x_{2})$ of the plane with the complex number $z=x_{1}+ix_{2}$. The Cauchy transform of a density $q$ defined on $\gamma$ is given by:
$$\mathscr C_\gamma q(z):=\frac{1}{2i\pi}\int_{\gamma}\frac{q(\zeta)}{\zeta-z}\,{\rm d}\zeta,\qquad z\in \mathbb C\setminus\gamma.$$
\begin{remark}\label{rem:Cauchyf}
It is worth noticing that the Cauchy transform can be easily computed in the following particular cases using Cauchy integral formula and  Cauchy integral theorem. 
\begin{enumerate}
\item
If $q$ is the trace of a holomorphic function $F$ on $\omega$, then
$$
\mathscr C_{\gamma} q(z)=
\begin{cases} 
F(z)&$\qquad z\in\omega$,\\
0&$\qquad z\in\omega^c:=\mathbb C \setminus\overline{\omega}$.
\end{cases}
$$
\item
If $q(z)= (z-a)^{-1}$, with  $a\notin\omega$, then 
$$
\mathscr C_{\gamma} q(z)=
\begin{cases} 
(z-a)^{-1}&$\qquad z\in\omega,$\\
0&$\qquad z\in\omega^c:=\mathbb C \setminus\overline{\omega}.$
\end{cases}
$$
\item
If $q(z)= (z-a)^{-1}$, with $a\in\omega$, then 
$$
\mathscr C_{\gamma} q(z)=
\begin{cases} 
0&$\qquad z\in\omega,$\\
-(z-a)^{-1}&$\qquad z\in\omega^c:=\mathbb C \setminus\overline{\omega}.$
\end{cases}
$$
\end{enumerate}
\end{remark}


It turns out (see for instance \cite[p.~254]{Hac95} or \cite[p.~100]{Kre99}) that for real-valued traces $q$, the double layer potential coincides with the real part of the Cauchy integral. Therefore, for complex-valued densities $q$, the Cauchy transform and the double layer potential are related via the following formulae:
\begin{subequations}
\label{magic}
\begin{align}
\mathscr D_\gamma (\Re q) &=\frac{1}{2}\Re\left(\mathscr C_\gamma(q+\bar q)\right)\\
\mathscr D_\gamma (\Im q) &=\frac{1}{2}\Im\left(\mathscr C_\gamma(q-\bar q)\right).
\end{align}
The above relations can be summarized in the identity:
\begin{equation}
\mathscr D_\gamma q=\frac{1}{2} \mathscr C_\gamma q+\frac{1}{2} \overline{\mathscr C_\gamma \bar q},
\end{equation}
or equivalently:
\begin{equation}
\mathscr D_\gamma \bar q=\frac{1}{2} \overline{\mathscr C_\gamma q}+\frac{1}{2} \mathscr C_\gamma \bar q.
\end{equation}
\end{subequations} 

\section{The reconstruction method}\label{sect:rec}

\subsection{From DtN measurements to GPST}\label{subsec:DtNGPST}
The first step of the proposed reconstruction method is to  recover the GPST of the cavities from the DtN measurements.
\begin{definition}
\label{def_K}
Let $\mathsf K_{\Gamma}^{\gamma}$ and $\mathsf K_{\gamma}^{\Gamma}$ be the operators:
$$
\mathsf K_{\Gamma}^{\gamma}: q\in  \mathcal H(\Gamma)\longmapsto {\mathsf{Tr}}^0_{\gamma}(\mathscr S_{\Gamma} {\hat q})\in  \mathcal H(\gamma),
\qquad
\mathsf K_{\gamma}^{\Gamma}: p\in  \mathcal H(\gamma)\longmapsto {\mathsf{Tr}}^0_{\Gamma}(\mathscr S_{\gamma} {\hat p})\in \mathcal H(\Gamma),
$$
where ${\mathsf{Tr}}^0_{\gamma}$ and ${\mathsf{Tr}}^0_{\Gamma}$ are given in Definition~\ref{def:Tr0}.
We define the boundary interaction operator between $\Gamma$ and $\gamma$ as the operator $\mathsf K:=\mathsf K_\gamma^\Gamma\mathsf K_\Gamma^\gamma$.
\end{definition}
One of the  main results of \cite{MunRam17} (see Theorem 3.1 therein ) is to provide a relation between the measurements and the boundary interaction operator $\mathsf K:=\mathsf K_\gamma^\Gamma\mathsf K_\Gamma^\gamma$ in the case of a single cavity. The preliminary results of Section \ref{sect:EI} immediately lead to a generalization of this relation to the case of multiple cavities (see Theorem \ref{first_decom} below). For the proofs, we refer the interested reader to \cite[Section 2.2.]{MunRam17}. 

\begin{proposition}\label{prop:KgG}
The boundary interaction operators $\mathsf K_{\Gamma}^{\gamma}$ and $\mathsf K_{\gamma}^{\Gamma}$ enjoy the following properties:
\begin{enumerate}
\item 
If $p\in \mathcal H(\gamma)$, then $q:={\mathsf Tr}_\Gamma(\SL{\gamma}\hat p)$ belongs to $\mathcal H(\Gamma)$.
\item 
Operators $\mathsf K_\Gamma^\gamma$ and $\mathsf K_\gamma^\Gamma$ are compact, one-to-one and dense-range operators. Moreover, for every functions $q\in \mathcal H(\Gamma)$ and $p\in \mathcal H(\gamma)$, we have:
\begin{equation}\label{eq:adjoint}
\langle \mathsf K_\Gamma^{\gamma} q,p\ranglepp{\gamma} = \langle  q,\mathsf K^\Gamma_\gamma p\ranglepp{\Gamma}.
\end{equation}
\item
The norms of the operators $\mathsf K_\Gamma^\gamma$ and $\mathsf K_\gamma^\Gamma$ are strictly less that 1.
\end{enumerate}
\end{proposition}
Note that the first assertion in Proposition \ref{prop:KgG}  shows that ${\mathsf{Tr}}^0_{\Gamma}$ can be replaced by ${\mathsf{Tr}}_{\Gamma}$ in the definition of $\mathsf K_{\gamma}^{\Gamma}$.\\
Going back to the DtN operator $\Lambda_\gamma$ of problem \eqref{main_problem}, and according to \eqref{free_circ}, we have by Green's formula
$$
\langle \partial_n u^f, 1\ranglemp{{\Gamma}} = -\langle \partial_n u^f, 1\ranglemp{{\gamma}}  =0,
$$
which shows that  $\Lambda_\gamma$ is valued in ${\widehat {\mathcal H}}({\Gamma})$. Considering data $f\in {\mathcal H}({\Gamma})$, we can thus define the DtN operator $\Lambda_\gamma$ as follows:
\begin{equation}
\label{D_to_N}
\Lambda_\gamma:f\in \mathcal H(\Gamma)\longmapsto \partial_nu^f\in {\widehat {\mathcal H}}({\Gamma}).
\end{equation}
Let us denote by $\Lambda_0$ the DtN map $\Lambda_\gamma$ in the case where  $\omega=\varnothing$ (the cavity free problem).
\begin{theorem}
\label{first_decom}
The two following bounded linear operators in $\mathcal H(\Gamma)$:
$$
\mathsf R:=\mathsf S_\Gamma(\Lambda_\gamma-\Lambda_0)\qquad\text{and}\qquad \mathsf K:=\mathsf K_\gamma^\Gamma\mathsf K_\Gamma^\gamma,
$$
satisfy the following equivalent identities:
\begin{equation}
\label{factor:main}
\mathsf R=({\rm Id}-\mathsf K)^{-1}{\mathsf K},
\qquad\qquad
\mathsf K=({\rm Id}+\mathsf R)^{-1}\mathsf R.
\end{equation}
\end{theorem}
The above result shows that the knowledge of the DtN maps $\Lambda_\gamma$ and $\Lambda_0$ respectively corresponding to the cases with and without the cavities, entirely determines the boundary interaction operator $\mathsf K$. Using \eqref{eq:adjoint}, it is worth reformulating the second identity in \eqref{factor:main} in a variational form:
\begin{equation}
\label{factor:main_varia}
\langle \mathsf K_\Gamma^\gamma f,\mathsf K_\Gamma^\gamma g\ranglepp{\gamma} = \langle ({\rm Id}+\mathsf R)^{-1}\mathsf R f,g\ranglepp{\Gamma},\qquad\forall\,f,g\in \mathcal H(\Gamma).
\end{equation}
This identity can be used to compute the entries of the so-called polarization tensors of the multiply connected  cavity. To make this statement precise, let us introduce the following definition. 
\begin{definition}\label{def:PQ}
Identifying $x=(x_{1},x_{2})$ in $\mathbb R^2$ with the complex number $z=x_1+ix_2$, we define for every $m\geqslant 1$,  the harmonic polynomials of degree $m$:
$$P_1^m(x)=\Re\left(z^m\right)\qquad\text{and}\qquad P_2^m(x)=\Im\left(z^m\right).$$
We define as well
\begin{equation}\label{eq:Qm12}
Q_{1,\Gamma}^m:= {\mathsf{Tr}}^0_{\Gamma} (P_1^m) \qquad\text{and}\qquad 
Q_{2,\Gamma}^m:= {\mathsf{Tr}}^0_{\Gamma} (P_2^m) 
\end{equation}
where the projected trace operator $ {\mathsf{Tr}}^0_{\Gamma}$ is defined in Definition \ref{def:Tr0} (recall that this projector simply amounts to adding a suitable constant to the considered function). \\
Finally, we set
\begin{equation}\label{eq:Qm}
Q_{\Gamma}^m:= Q_{1,\Gamma}^m+i Q_{2,\Gamma}^m.
\end{equation}
\end{definition}

The crucial point about these polynomials $Q_{\ell,\Gamma}^m$, $\ell=1,2$, lies in the fact that since they are traces of harmonic functions, we have
$$
\mathsf K_{\Gamma}^{\gamma}(Q_{\ell,\Gamma}^m) = Q_{\ell,\gamma}^m,
$$
and hence, applying formula \eqref{factor:main_varia} with $Q_{\ell,\Gamma}^m$, we obtain that for all $m,m'\geqslant 1$ and all $\ell,\ell'=1,2$:
$$\langle \mathsf K_\Gamma^\gamma Q_{\ell,\Gamma}^m,\mathsf K_\Gamma^\gamma Q^{m'}_{\ell',\Gamma}\ranglepp{\gamma}=\langle Q_{\ell,\gamma}^m, Q^{m'}_{\ell',\gamma}\ranglepp{\gamma}.
$$
These quantities are strongly connected with the so-called Generalized P\'olya-Szeg\"o Tensors (GPST) used in \cite{MunRam17} to reconstruct a single cavity. Unfortunately, the reconstruction method proposed there rests on the Riemann mapping theorem, which does not apply in the multiply connected case considered in this paper. We propose in the next section a new reconstruction method to recover the geometry of the cavities from the available data, namely the real quantities $\langle Q^m_{\ell,\gamma}, Q^{m'}_{\ell',\gamma} \rangle_{\frac 12, \gamma}$, or equivalently the complex quantities $\langle Q^m_{\gamma}, Q^{m'}_{\gamma} \rangle_{\frac 12, \gamma}$ (see \eqref{eq:psc}). In fact, as already pointed out in \cite{MunRam17}, the terms $\langle Q^m_{\gamma}, Q^{1}_{\gamma} \rangle_{\frac 12, \gamma}$ contain enough information to recover the unknown geometry.

\subsection{Towards a harmonic moments problem}
\label{subsec:moment}
The second step of the reconstruction is to recast the shape-from-GPST problem (namely reconstructing $\gamma$ from the quantities $\langle Q^m_{\gamma}, Q^{1}_{\gamma} \rangle_{\frac 12, \gamma}$) as a moments problem. Let us recall that the classical moments problem consists in recovering an unknown measure with support in $K\subset \mathbb C$ from its moments. The literature on this problem is very rich and covers a wide range of questions (solvability, uniqueness and reconstruction) and settings (dimension one or higher dimensions, arbitrary measures or measures absolutely continuous with respect to the Lebesgue measure, full or partial (harmonic or truncated) set of moments,...). Proposing a complete review is thus clearly beyond the scope of this paper. Let us simply make a few comments and quote some references. As far as we know, there is no framework to tackle this problem in full generality and most contributions address particular issues. The solvability of the moment problem in the cases $K=[0,+\infty[$, $K=\{|z|=1\}$, $K=\mathbb R$ and $K=[0,1]$ has been answered by the classical theorems of Stieltjes, Toeplitz, Hamburger and Hausdorff (see for instance Curto and Fialkow \cite{CurFia91} and references therein). In dimension two, Davis studied the reconstruction of a triangle from four moments \cite{Dav77}, Milanfar {\it et al.} \cite{MilVerCle95} investigated the reconstruction of arbitrary polygons and Putinar \cite{Put88} provided solvability conditions for the moments problem in the complex plane and pointed out some nice connections with quadrature domains (two-dimensional domains which are uniquely determined by finitely many of their moments, see \cite{Ebe05}). Finally, two-dimensional  shape reconstruction algorithms have been also proposed \cite{GolMilVar99,GusHeMil00,MilPutVar00,GusPutSaf09}.\\
For the problem studied in this paper, the connection between the shape-from-GPST problem and the  shape-from-moments problem is based on the existence of a Borel measure $\nu$ supported in the cavity $\omega$ such that ($z$ denotes  the variable in the complex plane)
\begin{equation}\label{eq:paramount}
\frac{1}{m} \langle Q^m_{\gamma}, Q^{1}_{\gamma} \rangle_{\frac 12, \gamma} = \int_{\omega} z^{m-1}\, {\rm d} \nu, \qquad \forall\,m\geqslant 1.
\end{equation}
We have been able to prove this formula only in two special cases: for a simply connected cavity and in the case of two disks. Still, the numerical results (see Section \ref{sect:num}) obtained for arbitrary cavities using the reconstruction method based on this formula (see subsection \ref{sect:algo}) are conclusive.

This is why we state the following conjecture.
\begin{conj}\label{conj}
Let $\omega=\cup_{k=1}^N\omega_k$ be a multiply connected set, where the sets $\omega_{k}$, for $k=1,\dots,N$ are non intersecting simply connected open domains with $C^{1,1}$  boundaries $\gamma_{k}$ and $\overline{\omega}\subset\Omega$. We set $\gamma=\cup_{k=1}^N\gamma_k$. Then, there exists a Borel measure $\nu$ supported in $\omega$ such that:
$$\langle f, Q^{1}_{\gamma}  \rangle_{\frac12,\gamma}=\int_\omega F'{\rm d}\nu,$$
for every trace $f\in \mathcal H^{\frac{1}{2}}(\gamma)$ of a holomorphic function $F$ in $\omega$. 
\end{conj} 
As mentioned above, we have obtained the following result, whose proof is given in Section \ref{sect:proof}. In particular, the expression of the measure $\nu$ is given therein. 
\begin{theorem}\label{thm:conj}
Conjecture \ref{conj} is true for $N=1$ (single cavity)  and when $\omega$ is constituted of two non intersecting disks.
\end{theorem}

The main interest of identity \eqref{eq:paramount} lies in the fact that it bridges two classical inverse problems, namely the Calder\`on's conductivity inverse problem and the historical moments problem. However, our moment problem has two features that makes it difficult to solve. First, we only have at our disposal the harmonic moments (i.e. the sequence $\int_{\omega} z^{m}\, {\rm d} \nu$, and not the doubly indexed sequence $\int_{\omega} z^{m}\bar z^{n}\, {\rm d} \nu$). Second, the involved measure is generally not of the form ${\rm d}\nu=\1{\omega}\, {\rm d}x$ which is the most studied in the literature. 

In order to have some insight on what this measure might represent, let us consider the particular case where the cavity is constituted of a collection of small inclusions. There is a wide literature dealing with this case and we refer the interested reader to the book by Ammari and Kang \cite{AmmKan04}. For small disks  of centers $z_{i}$ and  radii $\varepsilon \rho_{i}$, $i=1,\dots,N$ ($\varepsilon>0$ being a small parameter), and for the boundary conditions considered in this work, the DtN map denoted $\Lambda_\varepsilon$ admits the following asymptotic expansion (apply for instance \cite[Theorem 2.1]{MunRam17}, in the case of non moving disks):
$$
\Lambda_\varepsilon=\Lambda_0+\varepsilon^2\Lambda_2+O(\varepsilon^3),
$$
where 
\begin{equation}\label{devDTN}
\langle\Lambda_2 f_{\Gamma},g_{\Gamma}\ranglemp{\Gamma} =2\pi \sum_{i=1}^N \rho^2_i\, F'(z_i)\overline{G'(z_i)}.
\end{equation}
In the above formula, $f_{\Gamma},g_{\Gamma}\in \mathcal H^{\frac{1}{2}}(\Gamma)$ denote respectively the traces on $\Gamma$ of holomorphic functions $F, G$ defined in the cavity $\omega_{\varepsilon}$. 
With the notation of Theorem \ref{first_decom}, this expansion shows that $\mathsf R_{\varepsilon}=\mathsf S_\Gamma(\Lambda_\varepsilon-\Lambda_0) = \varepsilon^2\mathsf S_\Gamma\Lambda_2 +O(\varepsilon^3)$ and thus
$$ 
\mathsf K_{\varepsilon}=(\rm{Id}+\mathsf R_{\varepsilon})^{-1}\mathsf R_{\varepsilon} = \varepsilon^2\mathsf S_\Gamma\Lambda_2 +O(\varepsilon^3).
$$
Hence, we have
$$
\langle f_{\gamma_{\varepsilon}}, Q^{1}_{\gamma_{\varepsilon}}  \rangle_{\frac12,\gamma_{\varepsilon}} = 
\langle \mathsf K_{\varepsilon} f_{\Gamma}, Q^{1}_{\Gamma}\ranglepp{\Gamma} =
\varepsilon^2\langle \Lambda_2  f_{\Gamma}, Q^{1}_{\Gamma}\ranglemp{\Gamma} +O(\varepsilon^3)
$$
Thanks to \eqref{devDTN}, the above formula reads
$$
\langle f_{\gamma_{\varepsilon}}, Q^{1}_{\gamma_{\varepsilon}}  \rangle_{\frac12,\gamma_{\varepsilon}} = 
2\pi\varepsilon^2 \sum_{i=1}^N \rho^2_i\, F'(z_i)+O(\varepsilon^3)=
\int_\omega F'{\rm d}\nu+O(\varepsilon^3),
$$
provided we set 
$$
\nu=\sum_{i=1}^n c_i \delta_{z_i}, \qquad\mbox{with } c_{i}= 2\pi(\varepsilon\rho_i)^2.
$$
This formula suggests that the measure $\nu$ can be approximated by an atomic measure, involving the centers and the radii of the unknown disks. 

\subsection{The reconstruction algorithm}
\label{sect:algo}
Going back to the general case of an arbitrary cavity, it is thus natural to seek an approximation of the measure $\nu$ appearing in Conjecture \ref{conj} in the form $ \nu^*=\sum_{i=1}^n c_i \delta_{z_i}$ for some integer $n$, where the weights $c_{i}$ are positive and the points $z_{i}$ are distinct.  

To do so, we equalize the first $2n$ complex moments:
\begin{equation}\label{eq:prony}
\tau_{m}=\sum_{i=1}^n c_i z^m_i, \qquad \forall\,m=0,\cdots,2n-1
\end{equation}
where 
$$
\tau_{m}:=\int_{\omega} z^{m}\, {\rm d} \nu, \qquad \forall\,m=0,\cdots,2n-1.
$$
The non linear system \eqref{eq:prony} with $2n$ unknowns is usually referred to as Prony's system. To solve it, we follow the method proposed by Golub {\it et al.} \cite{GolMilVar99}. For all integer $0\leqslant m\leqslant 2n-1$, let:
$$
V = \begin{pmatrix} 1 & 1 &\cdots & 1 \\ z_1&z_2&\cdots &z_n\\
\vdots & \vdots & \ddots & \vdots \\ z_1^{n-1} & z_2^{n-1} & \cdots& z_{n}^{n-1}\end{pmatrix},
$$
and 
$$
H_{0}=\begin{pmatrix} \tau_0 & \tau_1 &\cdots & \tau_{n-1} \\ \tau_1& \tau_2&\cdots &\tau_n\\
\vdots & \vdots & \ddots & \vdots \\ \tau_{n-1} & \tau_n & \cdots& \tau_{2n-2}\end{pmatrix},
\qquad
H_{1}=\begin{pmatrix} \tau_1 & \tau_2 &\cdots & \tau_{n} \\ \tau_2& \tau_3&\cdots &\tau_{n+1}\\
\vdots & \vdots & \ddots & \vdots \\ \tau_{n} & \tau_{n+1} & \cdots& \tau_{2n-1}\end{pmatrix}.
$$
Setting $\mathbf c=(c_{i})_{1\leqslant i\leqslant n}$ and $\mathbf z=(z_{i})_{1\leqslant i\leqslant n}$, we easily see that if the pair $(\mathbf z,\mathbf c)$ solves System \eqref{eq:prony}, then the points $z_{i}$, $1\leqslant i\leqslant n$, are the eigenvalues of the generalized eigenvalue problem
\begin{equation}\label{eq:vpg}
H_{1}\mathbf x = \lambda\,H_{0}\mathbf x.
\end{equation}
Once the $(z_{i})_{1\leqslant i\leqslant n}$ have been determined, the weights $(c_{i})_{1\leqslant i\leqslant n}$ can be easily computed (by solving the linear system \eqref{eq:prony} or computing the diagonal elements of the matrix $V^{-1}H_{0}V^{-T}$).

Considering the particular case of small disks described in subsection \ref{subsec:moment}, the cavity will be reconstructed by drawing the disks of centers $(z_{i})_{1\leqslant i\leqslant n}$ and radii $(\rho_{i})_{1\leqslant i\leqslant n}$, with $\rho_{i}=\sqrt{|c_{i}|/2\pi}$ (see Section \ref{sect:num} for some examples).

\section{Proof of Theorem \ref{thm:conj}}
\label{sect:proof}

\subsection{Connected obstacle of arbitrary shape}
If the cavity has one connected component, then according to  the Riemann mapping theorem, its boundary $\gamma$ can be parameterized by  $t\in]-\pi,\pi]\mapsto \phi(e^{it})$, where
$$
\phi:z\mapsto a_1 z+a_0+\sum_{m\leqslant -1} a_m z^m,
$$
maps the exterior of the unit disk $\mathbb D$ onto the exterior of $\omega$ (see the book of Pommerenke \cite[p.~5]{Pom92}). Let us emphasize that $|a_1|$ is the logarithmic capacity of $\gamma$ and can be chosen such that $a_1>0$, while $a_0$ is  the conformal center of $\omega$. According to \cite[Lemma 3.5]{MunRam17} (see formula (3.11a)), we have
$$
\langle Q^m_{\gamma}, Q^{1}_{\gamma} \rangle_{\frac 12, \gamma} = 2a_1\int_{-\pi}^\pi e^{-it}\phi^m(e^{it})\,{\rm d}t,\qquad m\geqslant 1.
$$
This formula can be easily generalized for every trace $f\in \mathcal H^{\frac{1}{2}}(\gamma)$ of a holomorphic function $F$ in $\omega$ as follows:
\begin{equation}\label{eq:fQ1}
\langle f, Q^{1}_{\gamma} \rangle_{\frac 12, \gamma} = 2a_1\int_{-\pi}^\pi e^{-it}F(\phi(e^{it}))\,{\rm d}t.
\end{equation}
\begin{theorem}\label{thm:conj1}
Assume that $\omega$ is a simply connected domain with $C^{1,1}$  boundary $\gamma$. Then, there exists a $C^\infty$ diffeomorphism $\eta$ that maps $\omega$ onto the unit disk $\mathbb D $ such that:
$$\langle f, Q^{1}_{\gamma} \rangle_{\frac 12, \gamma}=4a_1^2\int_\omega F'(\xi) |D\eta|\,|{\rm d}\xi|.$$
In particular, Conjecture \ref{conj} is true for $N=1$ with ${\rm d}\nu=4a_1^2 |D\eta|\,|{\rm d}\xi|$.
\end{theorem}

\begin{remark}
Note that the diffeomorphism $\eta$ is not a conformal mapping (not even a harmonic function) and satisfies $\eta(\xi)=\phi^{-1}(\xi)$ on $\gamma=\partial\omega$.
\end{remark}

\begin{proof}[Proof of Theorem \ref{thm:conj1}]
Integrating by parts in \eqref{eq:fQ1}, we get that:
$$\langle f, Q^{1}_{\gamma} \rangle_{\frac 12, \gamma}=2a_1\int_{-\pi}^\pi \phi'(e^{it})F'(\phi(e^{it}))\,{\rm d}t=2a_1\int_{-\pi}^\pi \frac{\phi'(e^{it})}{|\phi'(e^{it})|}F'(\phi(e^{it}))\,|\phi'(e^{it})|{\rm d}t.$$
Denoting by $\psi=\phi^{-1}$ the map that conformally sends the exterior of $\omega$ onto the exterior of the unit disk $\mathbb D$, we obtain:
$$\langle f, Q^{1}_{\gamma} \rangle_{\frac 12, \gamma}=2a_1\int_\gamma \frac{|\psi'(\xi)|}{\psi'(\xi)}F'(\xi){\rm d}|\xi|=2a_1\int_\gamma \overline{\psi(\xi)}F'(\xi)\psi(\xi)\frac{|\psi'(\xi)|}{\psi'(\xi)}{\rm d}|\xi|,$$
where we have used the fact that $|\psi(\xi)|=1$. In the right hand side, we recognize:
$$\psi(\xi)\frac{|\psi'(\xi)|}{\psi'(\xi)} = n_1(\xi)+in_2(\xi),$$
where $n=(n_1,n_2)$ is the unit normal to $\gamma$. We deduce that:
\begin{align*}
\Re\left(F'(\xi)\psi(\xi)\frac{|\psi'(\xi)|}{\psi'(\xi)}\right)=\partial_n\Re(F(\xi))\\
\Im\left(F'(\xi)\psi(\xi)\frac{|\psi'(\xi)|}{\psi'(\xi)}\right)=\partial_n\Im(F(\xi)).
\end{align*}
Hence
\begin{multline*}
\hspace{2cm}
\frac{1}{2a_1}\langle f, Q^{1}_{\gamma} \rangle_{\frac 12, \gamma}=\int_\gamma \Re(\psi(\xi))\partial_n\Re(F(\xi))+\Im(\psi(\xi))\partial_n\Im(F(\xi)){\rm d}|\xi|\\
+i\int_\gamma \Re(\psi(\xi))\partial_n\Im(F(\xi))-\Im(\psi(\xi))\partial_n\Re(F(\xi)){\rm d}|\xi|.\hspace{2cm}
\end{multline*}
At this point, we would like to integrate by parts but $\psi$ is not defined inside $\omega$. We introduce the harmonic  (but non holomorphic) diffeomorphism:
$$\delta(z,\bar z)=a_1z+a_0+a_{-1}\bar z+a_{-2}\bar z^2+\ldots$$
defined from $\mathbb D$ onto $\omega$ and such that 
$$\delta(z,\bar z)=\phi(z)\text{ for }z\in\partial\mathbb D.$$
We define as well its inverse $\eta = \delta^{-1}$ (which in not even harmonic  but only $C^\infty$) and we have:
$$\eta(\xi,\bar \xi)=\psi(\xi)\text{ for }\xi\in\gamma.$$
Tolerating  a slight abuse of notation, we shall write $\eta(x,y)=\eta_1(x,y)+i\eta_2(x,y)=\eta(\xi,\bar\xi)$, for $\xi=x+iy$. We can now integrate by parts to get:
\begin{multline*}
\hspace{2cm}
\frac{1}{2a_1}\langle f, Q^{1}_{\gamma} \rangle_{\frac 12, \gamma}=\int_\omega\nabla \eta_1(x,y)\cdot\nabla\Re(F(\xi))+\nabla \eta_2(x,y)\cdot\nabla\Im(F(\xi))|{\rm d}\xi|\\
+i\int_\omega \nabla \eta_1(x,y)\cdot\nabla\Im(F(\xi))-\nabla \eta_2(x,y)\cdot\nabla\Re(F(\xi))|{\rm d}\xi|.
\hspace{2cm}
\end{multline*}
Recall that $\nabla\Re(F(\xi))=(\Re(F'(\xi)),-\Im(F'(\xi)))$ and $\nabla\Im(F(\xi))=(\Im(F'(\xi)),\Re(F'(\xi)))$. We deduce that:
\begin{multline*}
\hspace{2cm}
\frac{1}{2a_1}\langle f, Q^{1}_{\gamma} \rangle_{\frac 12, \gamma}=\int_\omega\Re(F'(\xi))\left(\frac{\partial\eta_1}{\partial x}+\frac{\partial\eta_2}{\partial y}\right)+\Im(F'(\xi))\left(-\frac{\partial\eta_1}{\partial y}+\frac{\partial\eta_2}{\partial x}\right)|{\rm d}\xi|\\
+i\int_\omega\Re(F'(\xi))\left(\frac{\partial\eta_1}{\partial y}-\frac{\partial\eta_2}{\partial x}\right)+\Im(F'(\xi))\left(\frac{\partial\eta_1}{\partial x}+\frac{\partial\eta_2}{\partial y}\right)|{\rm d}\xi|,
\hspace{2cm}
\end{multline*}
which can be simply rewritten as:
$$\frac{1}{2a_1}\langle f, Q^{1}_{\gamma} \rangle_{\frac 12, \gamma}=\int_\omega F'(\xi) \left[\left(\frac{\partial\eta_1}{\partial x}+\frac{\partial\eta_2}{\partial y}\right)+i\left(\frac{\partial\eta_1}{\partial y}-\frac{\partial\eta_2}{\partial x}\right)\right]|{\rm d}\xi|.$$
On the one hand, we have:
\begin{align*}
\frac{\partial\eta}{\partial x}&=\frac{\partial\eta_1}{\partial x}+i\frac{\partial\eta_2}{\partial x}\\
\frac{\partial\eta}{\partial y}&=\frac{\partial\eta_1}{\partial y}+i\frac{\partial\eta_2}{\partial y}.
\end{align*}
Hence, we deduce that (see \cite[p.~3]{Dur04})
\begin{align*}
\frac{\partial\eta_1}{\partial x}+\frac{\partial\eta_2}{\partial y}&=\Re\left(\left(\frac{\partial}{\partial x}-i\frac{\partial }{\partial y}\right)\eta\right)=2\Re\left(\frac{\partial\eta}{\partial\xi}\right)\\
\frac{\partial\eta_1}{\partial y}-\frac{\partial\eta_2}{\partial x}&=\Re\left(\left(\frac{\partial}{\partial y}+i\frac{\partial }{\partial x}\right)\eta\right)=-\Im\left(\left(\frac{\partial}{\partial x}-i\frac{\partial }{\partial y}\right)\eta\right)=-2\Im\left(\frac{\partial\eta}{\partial\xi}\right),
\end{align*}
and therefore:
\begin{equation}\label{eq:mumeta}
\frac{1}{4a_1}\langle f, Q^{1}_{\gamma} \rangle_{\frac 12, \gamma}=\int_\omega F'(\xi) \overline{\left(\frac{\partial\eta}{\partial\xi}\right)}|{\rm d}\xi|.
\end{equation}
On the other hand, using the chain rule formulae (see \cite[p.~5]{Dur04}), we have:
\begin{align*}
\frac{\partial}{\partial z}(\eta\circ\delta)&=\frac{\partial\eta}{\partial\xi}\frac{\partial\delta}{\partial z}+\frac{\partial\eta}{\partial\bar\xi}\overline{\frac{\partial\delta}{\partial\bar z}}=1\\
\frac{\partial}{\partial \bar z}(\eta\circ\delta)&=\frac{\partial\eta}{\partial\xi}\frac{\partial\delta}{\partial \bar z}+\frac{\partial\eta}{\partial\bar\xi}\overline{\frac{\partial\delta}{\partial z}}=0.
\end{align*}
Eliminating $\dfrac{\partial\eta}{\partial\bar\xi}$ from the above relations, we get that
\begin{equation}\label{eq:etadelta}
\left(\left|\frac{\partial\delta}{\partial z}\right|^2-\left|\frac{\partial\delta}{\partial \bar z}\right|^2\right)\frac{\partial\eta}{\partial\xi}=\overline{\frac{\partial\delta}{\partial z}}.
\end{equation}
But, on the one hand, a direct calculation shows that $\dfrac{\partial\delta}{\partial z}=a_1$, while on the other hand, we know (see \cite[p.~5]{Dur04}) that for a harmonic diffeomorphism, we have: 
$$\left|\frac{\partial\delta}{\partial z}\right|^2-\left|\frac{\partial\delta}{\partial \bar z}\right|^2=|D\delta|=|D\eta|^{-1}.$$
Substituting these relations in \eqref{eq:etadelta}, we obtain that 
$$\frac{\partial\eta}{\partial\xi}=a_1|D\eta|.$$
The conclusion follows then from \eqref{eq:mumeta}.
\end{proof}

\subsection{The case of two disks}
We assume here that $\omega=\omega_{1}\cup \omega_{2}$, where $\omega_1$ and $\omega_2$ are two disks centered at $z_1$ and $z_2$ respectively and with radii $\rho_1$ and $\rho_2$. We set $\rho:=|z_1-z_2|$ and we denote by $\omega_j^c=\mathbb C\setminus\overline{\omega_j}$  ($j=1,2$).
\begin{proposition}\label{prop:zkappaq}
There exists two sequences of points $(z_{1,n})_{n\geqslant 0}$ and $(z_{2,n})_{n\geqslant 0}$ and four sequences of complex numbers $(\kappa_{1,n})_{n\geqslant 1}$, $(\kappa_{2,n})_{n\geqslant 1}$, $(\alpha_{1,n})_{n\geqslant 0}$ and $(\alpha_{2,n})_{n\geqslant 0}$ such that the following properties hold true:
\begin{enumerate}
\item $z_{1,n}\in \omega_1$  and $z_{2,n}\in \omega_2$ for all $n\geqslant 0$.
\item The densities defined by
$$
q_{j,0}(z)=2(z-z_j),
$$
and for all $n\geqslant 1$:
$$
q_{1,n}(z)=\frac{\kappa_{1,n}}{z-z_{2,n-1}} +\alpha_{1,n},
\qquad\qquad
q_{2,n}(z)=\frac{\kappa_{2,n}}{z-z_{1,n-1}} +\alpha_{2,n},
$$
are such that for $j=1,2$ the series $\sum_{n\geqslant 0} q_{j,2n} + \overline{q}_{j,2n+1}$ converges in $H^{\frac12}(\gamma_j)$ to some limit $q_{j}$.
\item The limits $q_1\in H^{\frac12}(\gamma_1)$ and $q_2\in H^{\frac12}(\gamma_2)$ are such that
$$
\mathscr D_{\gamma_1}q_1(\xi)+\mathscr D_{\gamma_2}q_2(\xi)=
\begin{cases}
\xi-z_1&\text{in }$\omega_1$\\
\xi-z_2&\text{in }$\omega_2.$
\end{cases}
$$
\end{enumerate}
\end{proposition}
The proof of this result is constructive: the sequences $(z_{j,n})_{n\geqslant 0}$, $(\kappa_{j,n})_{n\geqslant 1}$ and $(\alpha_{j,n})_{n\geqslant 0}$ ($j=1,2$) are respectively given by formulae \eqref{eq:zjn}, \eqref{eq:kappajn} and \eqref{eq:alphajn}. In order to prove the above Proposition, let us  introduce some further notation and prove two lemmas. Setting:
$$\Lambda_j=\left(\frac{\rho_j}{\rho}\right)^2\qquad j=1,2$$
let us define the two real sequences $(\lambda_{1,n})_{n\geqslant 0}$ and $(\lambda_{2,n})_{n\geqslant 0}$ by:
$$
\lambda_{1,0}=\lambda_{2,0}=1,
$$
and for $n\geqslant 1$
$$
\lambda_{1,n}=1-\frac{\Lambda_2}{\lambda_{2,n-1}}\qquad\qquad
\lambda_{2,n}=1-\frac{\Lambda_1}{\lambda_{1,n-1}},
$$
It is worth noticing that 
$$
\lambda_{1,n}=1-\frac{\Lambda_2}{1-\displaystyle\frac{\Lambda_1}{\lambda_{1,n-2}}}.
$$
\begin{lemma}
\label{lem:conv}
The real sequences $(\lambda_{1,n})_{n\geqslant 0}$ and $(\lambda_{2,n})_{n\geqslant 0}$ are convergent. More precisely, we have:
$$\lim_{n\to +\infty}\lambda_{1,n}=\frac{1}{2}\left[1+\Lambda_1-\Lambda_2+\sqrt{\Lambda}\right]\quad\text{and}\quad
\lim_{n\to +\infty}\lambda_{2,n}=\frac{1}{2}\left[1+\Lambda_2-\Lambda_1+\sqrt{\Lambda}\right],$$
where $\Lambda: = |\Lambda_1- \Lambda_2|^2+1-2(\Lambda_1+\Lambda_2)>0$.
Moreover, for every $n\geqslant 0$:
\begin{align*}
1-\sqrt{\Lambda_2}&<\frac{1}{2}\left[1+\Lambda_1-\Lambda_2+\sqrt{\Lambda}\right]<\lambda_{1,n}\leqslant 1,\\
1-\sqrt{\Lambda_1}&<\frac{1}{2}\left[1+\Lambda_2-\Lambda_1+\sqrt{\Lambda}\right]<\lambda_{2,n}\leqslant 1.
\end{align*}
\end{lemma}
\begin{proof}
The convergence of the two sequences $(\lambda_{j,n})_{n\geqslant 0}$ ($j=1,2$) depends on the sign of $\Lambda$. Noticing that $\Lambda_2/\Lambda_1=(\rho_2/\rho_1)^2$ and rewriting $\Lambda$ as:
$$\Lambda=\Lambda_2^2\left(\frac{\rho_1^2}{\rho_2^2}-1\right)^2-2\Lambda_2\left(\frac{\rho_1^2}{\rho_2^2}+1\right)+1,$$
we deduce that:
$$\Lambda=\left(1-\frac{(\rho_1-\rho_2)^2}{\rho^2}\right)\left(1-\frac{(\rho_1+\rho_2)^2}{\rho^2}\right),$$
and therefore that $\Lambda>0$. \\
Since $(1-\Lambda_1+\Lambda_2)>0$  and $\Lambda=( 1-\Lambda_1+\Lambda_2)^2-4\Lambda_2$, we have:
$$\sqrt{\Lambda}-(1-\Lambda_1+\Lambda_2)<0.$$
It follows that:
$$(1-\Lambda_1+\Lambda_2)^2-4\Lambda_2+\Lambda-2(1-\Lambda_1+\Lambda_2)\sqrt{\Lambda}=2\sqrt{\Lambda}(\sqrt{\Lambda}-(1-\Lambda_1+\Lambda_2))<0.$$
We deduce first that:
$$4\Lambda_2>(1-\Lambda_1+\Lambda_2)^2+\Lambda-2(1-\Lambda_1+\Lambda_2)\sqrt{\Lambda}=(1-\Lambda_1+\Lambda_2-\sqrt{\Lambda})^2,$$
and hence that 
$$1-\sqrt{\Lambda_2}<\frac{1}{2}\left[1+\Lambda_1-\Lambda_2+\sqrt{\Lambda}\right].$$
\end{proof}
Let us define the two complex sequences:
\begin{equation}
\label{eq:zjn}
z_{1,n}=(1-\lambda_{2,n})z_2+\lambda_{2,n} z_1,\qquad\qquad z_{2,n}=(1-\lambda_{1,n})z_1+\lambda_{1,n} z_2.
\end{equation}
We also introduce the complex sequences $(\kappa_{1,n})_{n\geqslant 1}$ defined by:
\begin{equation}\label{eq:kappajn}
\kappa_{1,1}=2\rho_2^2 \qquad  \kappa_{1,2}=-\frac{2(\rho_1\rho_2)^2}{(\overline{z_2-z_1})^2},
\qquad  
\kappa_{1,{n+2}}=\frac{\Lambda_1\Lambda_2}{(\lambda_{2,n}\lambda_{1,n-1})^2}\kappa_{1,n},  \qquad n\geqslant 1.
\end{equation}
Note that we  have the following relations for all $n\geqslant 1$:
\begin{equation}
\label{nice_one}
\kappa_{1,{n+2}}=\frac{(1-\lambda_{1,n+1})(1-\lambda_{2,n})}{\lambda_{1,n-1}\lambda_{2,n}}\kappa_{1,n}=\frac{\Lambda_1\Lambda_2}{\left(\lambda_{1,n-1}-\Lambda_1\right)^2}\kappa_{1,n}.
\end{equation}
We define similarly a sequence $(\kappa_{2,n})_n$ by exchanging the role of the boundary indices. The sequences $(\kappa_{1,n})_n$ and $(\kappa_{2,n})_n$ can be deduced from each other according to the following identities:
\begin{align}
\kappa_{1,n+1}&=-\frac{\rho_2^2}{(\overline{z_2-z_1})^2\lambda_{2,n-1}^2}\overline{\kappa}_{2,n}=-\frac{(z_2-z_1)^2}{|z_2-z_1|^2}\frac{(1-\lambda_{1,n})}{\lambda_{2,n-1}}\overline{\kappa}_{2,n}\\
\kappa_{2,n+1}&=-\frac{\rho_1^2}{(\overline{z_1-z_2})^2\lambda_{1,n-1}^2}\overline{\kappa}_{1,n}=-\frac{(z_1-z_2)^2}{|z_2-z_1|^2}\frac{(1-\lambda_{2,n})}{\lambda_{1,n-1}}\overline{\kappa}_{1,n}.\label{eq:kappa2np1}
\end{align}
\begin{lemma}
\label{estim_Q}
There exists a constant $C>0$ and a constant $0<\delta<1$ such that:
$$|\kappa_{1,{n}}|+|\kappa_{2,{n}}|\leqslant C\delta^n\qquad\qquad\forall n\geqslant 1.$$
\end{lemma}
\begin{proof}
According to Lemma~\ref{lem:conv} and to relation \eqref{nice_one}, it suffices to prove that:
$$0<\frac{4\Lambda_1\Lambda_2}{\left(1-\Lambda_1-\Lambda_2+\sqrt{\Lambda}\right)^2}<1.$$
Direct computations lead to:
$$(1-\Lambda_1-\Lambda_2+\sqrt{\Lambda})^2-4\Lambda_1\Lambda_2=2\sqrt{\Lambda}\left[\sqrt{\Lambda}+(1-\Lambda_1-\Lambda_2)\right],$$
and the quantity in the right hand side is clearly positive. The conclusion follows.
\end{proof}
We are now in position to prove the claims detailed in Proposition \ref{prop:zkappaq}. 
\begin{proof}[Proof of Proposition \ref{prop:zkappaq}] \ \\
1. According to the estimate of Lemma~\ref{lem:conv}, we have:
$$|z_{2,n}-z_2|=\rho (1-\lambda_{1,n})<\rho\sqrt{\Lambda_2}=\rho_2,$$
and therefore, the point $z_{2,n}\in\omega_2$ for every $n\geqslant 0$. Of course, similar arguments show that $z_{1,n}\in\omega_1$.\ \\
2. Let the densities $q_{1,n}$ be defined as in Proposition \ref{prop:zkappaq}, with
\begin{equation}\label{eq:alphajn}
\alpha_{1,n}=- \frac12\frac{\kappa_{1,n}}{\lambda_{1,n-1}(z_1-z_2)}.
\end{equation}
According to lemmas \ref{lem:conv} and \ref{estim_Q}, the series of functions:
$$
\sum_{n\geqslant 0}\left\{\frac{\kappa_{1,2n}}{z-z_{2,2n-1}}+\alpha_{1,2n}\right\}
+\left\{\frac{\overline{\kappa}_{1,2n+1}}{\overline{z-z_{2,2n}}}+ \overline{\alpha}_{1,2n+1}\right\},
$$
converges in $H^1(\omega_1)$ and hence its trace on $\gamma_1$ converges in $H^{\frac12}(\gamma_1)$ to some limit $q_1$. The same reasoning holds for $j=2$. \ \\
3. According to Remark \ref{rem:Cauchyf}, the densities $q_{j,0}$ and $\bar q_{j,0}$ ($j=1,2$) admit the following Cauchy transforms:
$$\mathscr C_{\gamma_j} q_{j,0}(\xi)=
\begin{cases} 
2(\xi-z_j)&\text{in }$\omega_j$\\
0&\text{in }$\omega^c_j$
\end{cases}
\qquad\text{and}\qquad
\mathscr C_{\gamma_j} \bar q_{j,0}(\xi)=
\begin{cases} 
0&\text{in }$\omega_j$\\
-\displaystyle\frac{2\rho_j^2}{\xi-z_j}&\text{in }$\omega^c_j,$
\end{cases}
$$
where, for the last equation, we have used the relation:
$$\overline{\xi-z_j}=\frac{\rho_j^2}{\xi-z_j}\qquad\text{on }\gamma_j.$$
Based on formulae \eqref{magic}, we deduce that:
$$\mathscr D_{\gamma_j} q_{j,0}(\xi)=
\begin{cases} 
\xi-z_j&\text{in }$\omega_j$\\
-\displaystyle\frac{\rho_j^2}{\overline{\xi-z_j}}&\text{in }$\omega^c_j.$
\end{cases}
$$
In particular, we have
$$
\mathscr D_{\gamma_1} q_{1,0}(\xi)+\mathscr D_{\gamma_2} q_{2,0}(\xi)=
\begin{cases}
\displaystyle(\xi-z_1)-\frac{\rho_2^2}{\overline{\xi-z_2}}&\text{in }$\omega_1$\\
\displaystyle(\xi-z_2)-\frac{\rho_1^2}{\overline{\xi-z_1}}&\text{in }$\omega_2.$
\end{cases}
$$
Similarly, and using once again Remark \ref{rem:Cauchyf}, we have  for every $n\geqslant 1$:
\begin{equation}\label{eq:Cq1n}
\mathscr C_{\gamma_1} q_{1,n}(\xi) =  
\begin{cases}
\displaystyle\frac{\kappa_{1,n}}{\xi-z_{2,n-1}}+\alpha_{1,n}&\text{in }$\omega_1$\\
0&\text{in }$\omega_1^c.$
\end{cases}
\end{equation}
In order to compute $\mathscr C_{\gamma_1} \bar q_{1,n}(\xi)$, we note that for every $\xi\in \gamma_1$, we have:
\begin{align}
\frac{1}{\overline{\xi-z_{2,n-1}}}&=\frac{1}{\frac{\rho_1^2}{\xi-z_1}+\overline{z_1-z_{2,n-1}}}\\
&=\frac{1}{\overline{z_1-z_{2,n-1}}}\left[1-\left(\frac{\rho_1^2}{\overline{z_1-z_{2,n-1}}}\right)\frac{1}{\xi-\left(z_1-\frac{\rho_1^2}{\overline{z_1-z_{2,n-1}}}\right)}\right].\label{eq:fracinv}
\end{align}
But notice now that, on the one hand:
$$\overline{z_1-z_{2,n-1}}=\lambda_{1,n-1}(\overline{z_1-z_2}),$$
and on the other hand:
\begin{align*}
z_1-\frac{\rho_1^2}{\overline{z_1-z_{2,n-1}}}&=z_1-\frac{\rho_1^2}{\lambda_{1,n-1}(\overline{z_1-z_2})}\\
&=z_1-\frac{\rho_1^2}{\rho^2\lambda_{1,n-1}}(z_1-z_2)\\
&=\left(1-\frac{\Lambda_{1}}{\lambda_{1,n-1}}\right)z_1+\frac{\Lambda_{1}}{\lambda_{1,n-1}}z_2\\
&=\lambda_{2,n}z_1+(1-\lambda_{2,n})z_2\\
&=z_{1,n}.
\end{align*}
Plugging the last two relations in \eqref{eq:fracinv} and using the identity  
$$\displaystyle\frac{\rho_1^2}{(\lambda_{1,n-1})^2(\overline{z_1-z_2})^2}\overline{\kappa}_{1,n}=-\kappa_{2,n+1},$$
we obtain that for $\xi\in \gamma_1$:
$$
\frac{\overline{\kappa}_{1,n}}{\overline{\xi-z_{2,n-1}}}=-2\overline{\alpha}_{1,n}
+\frac{\kappa_{2,n+1}}{\xi-z_{1,n}}.
$$
Hence
$$
\bar q_{1,n}(\xi) = \frac{\kappa_{2,n+1}}{\xi-z_{1,n}}-\overline{\alpha}_{1,n}.
$$
We deduce first that:
$$\mathscr C_{\gamma_1} \overline{q}_{1,n}(\xi) =  
\begin{cases}
\displaystyle-\overline{\alpha}_{1,n} &\text{in }$\omega_1$\\[8pt]
-\displaystyle\frac{\kappa_{2,n+1}}{\xi-z_{1,n}}&\text{in }$\omega_1^c,$
\end{cases}
$$
and next, using \eqref{eq:Cq1n}, that:
$$
\mathscr D_{\gamma_1} {q}_{1,n}(\xi) =  
\frac12 \mathscr C_{\gamma_1} q_{1,n}(\xi)
+\frac12\overline{\mathscr C_{\gamma_1} \overline{q}_{1,n}(\xi)}=
\begin{cases}
\displaystyle \frac{1}{2}\frac{\kappa_{1,n}}{\xi-z_{2,n-1}}&\text{in }$\omega_1$\\[8pt]
\displaystyle -\frac{1}{2}\frac{\overline{\kappa}_{2,n+1}}{\overline{\xi-z_{1,n}}}&\text{in }$\omega_1^c.$
\end{cases}
$$
After similar computations for $\mathscr D_{\gamma_1} {q}_{2,n}$, we finally obtain that
$$\mathscr D_{\gamma_1} {q}_{1,n}(\xi)+\mathscr D_{\gamma_2} {q}_{2,n}(\xi)
=
\begin{cases}
\displaystyle\frac{1}{2}\,\frac{\kappa_{1,n}}{\xi-z_{2,n-1}}-\frac{1}{2}\frac{\overline{\kappa}_{1,n+1}}{\overline{\xi-z_{2,n}}}&\text{in }$\omega_1$\\[8pt]
\displaystyle\frac{1}{2}\,\frac{\kappa_{2,n}}{\xi-z_{1,n-1}}-\frac{1}{2}\frac{\overline{\kappa}_{2,n+1}}{\overline{\xi-z_{1,n}}}&\text{in }$\omega_2.$
\end{cases}
$$
It is then easy to verify that:
$$
\sum_{k=0}^{n}\sum_{j=1,2}\mathscr D_{\gamma_j} {q}_{j,2k}(\xi)+\mathscr D_{\gamma_j} \overline{q}_{j,2k+1}(\xi)
=
\begin{cases}
\displaystyle (\xi-z_1)-\frac{1}{2}\frac{\kappa_{1,2n+2}}{{\xi-z_{2,2n+1}}}&\text{in }$\omega_1$\\[8pt]
\displaystyle (\xi-z_2)-\frac{1}{2}\frac{\kappa_{2,2n+2}}{{\xi-z_{1,2n+1}}}&\text{in }$\omega_2.$
\end{cases}
$$
Letting $n$ go to infinity and invoking again Lemma~\ref{estim_Q}, we obtain the claimed result.
\end{proof}
We can now prove the main results of this section.
\begin{theorem}\label{thm:quadrature2disks}
There exists two sequences of complex weights $(c_{1,n})_{n\geqslant 0}$ and $(c_{2,n})_{n\geqslant 0}$ such that for every $f=(f_1,f_2)\in H^\frac12(\gamma_1\cup\gamma_2)$ with $f_j$ the trace of a holomorphic function $F_j$ in $\omega_j$, there holds
$$\langle f, z\rangle_{\frac{1}{2},\gamma}=\sum_{j=1,2}\sum_{n\geqslant 0}c_{j,n}F'_j(z_{j,n}).$$
In other words, we have
$$
\langle f, z\rangle_{\frac{1}{2},\gamma}=\int_{\omega} F'{\rm d}\nu,
$$
where the measure $\nu$ is defined by:
$$\nu=\sum_{j=1,2}\sum_{n\geqslant 0}c_{j,n}\delta_{z_{j,n}}.$$
\end{theorem}

\begin{proof}
Let us first prove the following identity:
\begin{equation}\label{eq:zf1}
\langle f, z\rangle_{\frac{1}{2},\gamma}=\sum_{j=1,2}-i\int_{\gamma_j}\overline {q_j(z)} F'_j(z)\,{\rm d}z.
\end{equation}
We use here the notation of Lemma \ref{lem:dc2}. If $F_{j}=F^1_{j}+i F^2_{j}$, let $\widetilde F_{j}=\widetilde F^1_{j}+i \widetilde F^2_{j}$ where  $\widetilde F^\ell_{j}$ ($\ell=1,2$) is the harmonic conjugate of $F^\ell_{j}$. Obviously, since $F_{j}$ is harmonic in $\omega_{j}$, we have $\widetilde F_j^1=F_j^2$ and $\widetilde F_j^2=-F_j^1$ in $\omega_{j}$, and hence $\widetilde F_j = -iF_j$ in $\omega_{j}$. Similarly, we have $\widetilde {z}-z_j = -i(z-z_j)=-i\mathscr D_{\gamma}q=\mathscr D_{\gamma}(-iq)$ in $\omega_{j}$. Therefore, we have $\check q = -iq$. Formula \eqref{eq:pqcheckc} shows that
$$
\langle f, z\ranglepp{\gamma}=-\langle \partial_{n}\widetilde F,\check q\ranglemp{\gamma} =-\langle \partial_{n} F, q\ranglemp{\gamma}=-\sum_{j=1,2}\int_{\gamma_j}\overline {q_j(z)}\partial_{n} F_j(z)\,{\rm d}s_z.
$$
Since $\partial_{n} F_j(z)=\partial_{n} F^1_j(z)+i\partial_{n} F^2_j(z)= F'_{j}(z)n(z)$ and ${\rm d}z=in(z){\rm d}s_z$, the above formula immediately yields \eqref{eq:zf1}.\\
As $q_{j}=\sum_{n\geqslant 0} q_{j,2n} + \overline{q}_{j,2n+1}$, equation \eqref{eq:zf1} implies that
$$
\langle f,z\rangle_{\frac{1}{2},\gamma}=\sum_{j=1,2}\sum_{n\geqslant 0}-i\int_{\gamma_j}\overline {q_{j,2n}(z)} F'_j(z){\rm d}z-i\int_{\gamma_j}{q_{j,2n+1}(z)} F'_j(z){\rm d}z.
$$
The density $q_{j,2n+1}$ being the trace of a holomorphic function in $\omega_j$, the last terms in the right hand side vanishes. Let us focus on the term:
$$-i\int_{\gamma_1}\overline {q_{1,2n}(z)} F'_1(z){\rm d}z=-i\int_{\gamma_1} \overline{\kappa}_{1,2n}\frac{1}{\overline{z-z_{2,2n-1}}} F'_1(z){\rm d}z.$$
But for $z\in\gamma_1$, we have:
$$\frac{1}{\overline{z-z_{2,2n-1}}}=\frac{1}{\lambda_{1,n-1}(\overline{z_1-z_{2}})}\left[1-\frac{\rho_1^2}{\lambda_{1,2n-1}(\overline{z_1-z_{2}})}\frac{1}{z-z_{1,2n}}\right],$$
and therefore:
\begin{align*}
-i\int_{\gamma_1}\overline {q_{1,2n}(z)} F'_1(z){\rm d}z&=\frac{i\rho_1^2 \overline{\kappa}_{1,2n}}{\lambda^2_{1,2n-1}(\overline{z_1-z_{2}})^2}\int_{\gamma_1}\frac{1}{z-z_{1,2n}}F'_1(z){\rm d}z\\
&=-\frac{2\pi \rho_1^2 \overline{\kappa}_{1,2n}}{\lambda^2_{1,2n-1}(\overline{z_1-z_{2}})^2}F'_1(z_{1,2n})\\
&=2\pi \kappa_{2,2n+1}F'_1(z_{1,2n}),
\end{align*}
where the last equality follows from \eqref{eq:kappa2np1}. Adding the contribution corresponding to $\gamma_{2}$, we finally obtain that
$$
\langle f, z\ranglepp{\gamma}=2\pi  \sum_{j=1,2}\sum_{n\geqslant 0}\left\{\kappa_{2,2n+1}F'_1(z_{1,2n})+\kappa_{1,2n+1}F'_2(z_{2,2n})\right\},
$$
and the proof is complete.
\end{proof}

\section{Numerical tests}
\label{sect:num}
We collect in this section some numerical experiments illustrating the feasibility of the proposed reconstruction method. For the sake of clarity, we first sum up the different steps of the simple reconstruction algorithm:
\begin{enumerate}
\item Compute a numerical approximation of the operator $\mathsf R=\mathsf S_\Gamma(\Lambda_\gamma-\Lambda_0)$.
\item Fix an integer $n\geqslant 1$ and compute for $0\leqslant m\leqslant 2n-1$: 
$$
\tau_{m}:=\int_{\omega} z^{m}\, {\rm d} \nu =  \frac{1}{m+1} \langle Q^{m+1}_{\gamma}, Q^{1}_{\gamma} \rangle_{\frac 12, \gamma}=\frac{1}{m+1} \langle Q^{m+1}_{\Gamma}, ({\rm Id}+\mathsf R)^{-1}\mathsf R Q^1_\Gamma \rangle_{\frac 12, \Gamma}.
$$
\item Following the method described in \S.~\ref{sect:algo}, solve Prony's system:
$$
\sum_{i=1}^n c_i z^m_i = \tau_{m}, \qquad \forall\,m=0,\cdots,2n-1.
$$  
to determine  the positive weights $(c_{i})_{1\leqslant i\leqslant n}$ and (distinct) points $(z_{i})_{1\leqslant i\leqslant n}$. 
\item Plot the disks of centers $(z_{i})_{1\leqslant i\leqslant n}$ and radii $(\rho_{i})_{1\leqslant i\leqslant n}$, with $\rho_{i}=\sqrt{|c_{i}|/2\pi}$.
\end{enumerate}
We refer the interested reader to \cite[Section 4]{MunRam17} for technical details about the implementation of steps 1 and 2 of the algorithm. 
\par
For practical reconstructions, a natural question is how to determine the number $n$ of atoms (disks) to be used. From our numerical experiments, there is no clear answer to this issue. However, increasing $n$ generally yields reconstructions of better quality. This fact is illustrated in Figures \ref{fig:rectangle}, \ref{fig:trefle} and \ref{fig:multi} which show respectively examples of reconstructions  for a rectangular cavity, a clover shaped cavity and a multiply connected cavity (with three connected components). 

The following remarks are worth being mentioned:
\begin{itemize}
\item
Increasing the number of atoms from $n$ to $n+1$ does not result in just adding an additional disk. Indeed, this leads to a new Prony's system and hence, to a completely new distribution of disks. 
\item
For a given value of $n$, we can obtain atoms with zero radii which seem to be randomly distributed outside the cavity. See for instance Figure  \ref{fig:rectangle:22}, where  among the 22 atoms, only 13 have non zero radii.  
\item
Sometimes, spurious atoms can be observed (see Figure \ref{fig:pb:16}), but they disappear when $n$ increases (see Figure \ref{fig:pb:17}).
\end{itemize}

\begin{figure}[h]
\centering
	\subfigure[$n=1$]{\includegraphics[width=0.4\textwidth]{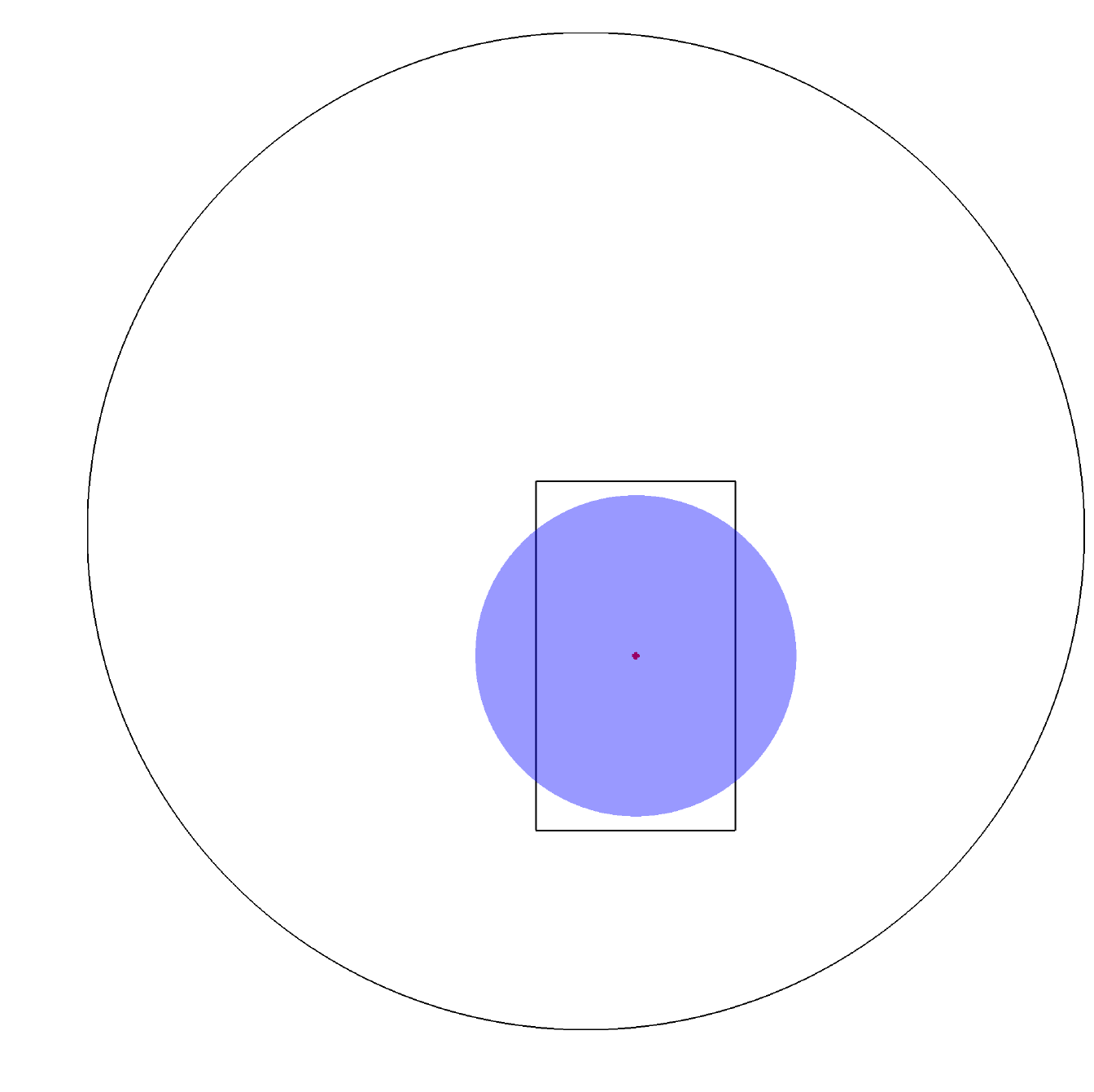}}\hspace{1cm}
	\subfigure[$n=5$]{\includegraphics[width=0.4\textwidth]{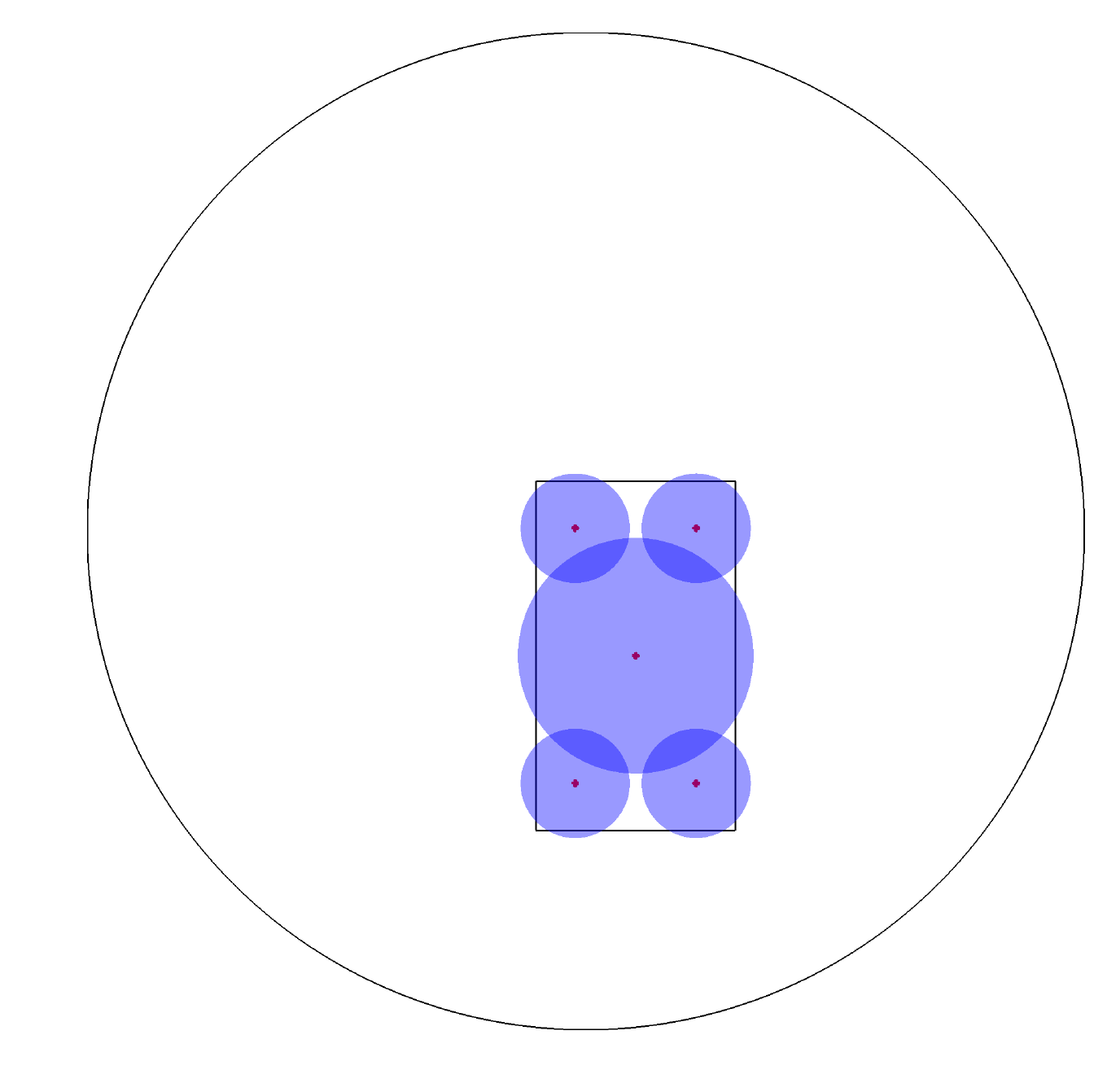}}\\
 	\subfigure[$n=10$]{\includegraphics[width=0.4\textwidth]{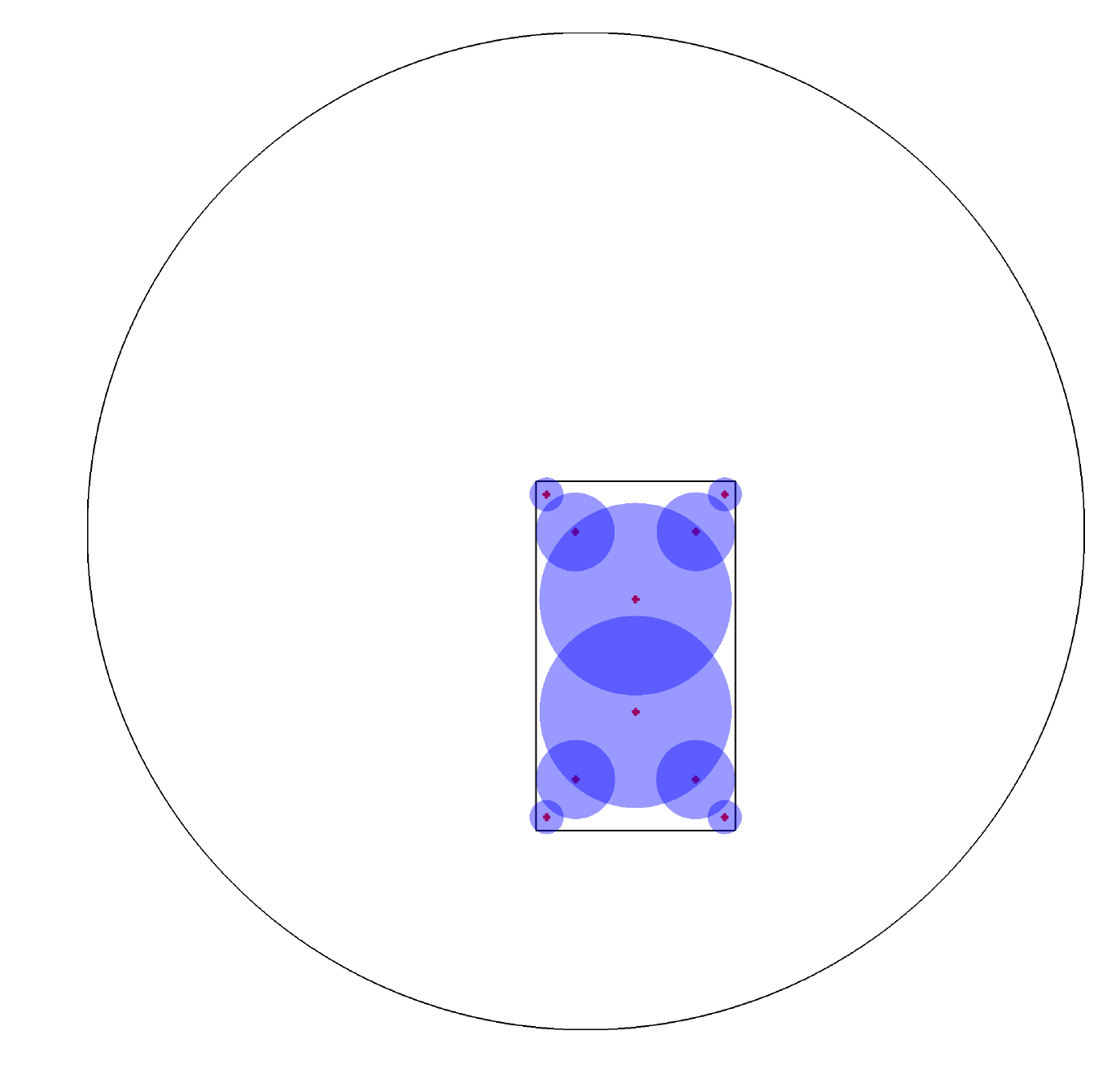}}\hspace{1cm}
	\subfigure[$n=22$\label{fig:rectangle:22}]{\includegraphics[width=0.4\textwidth]{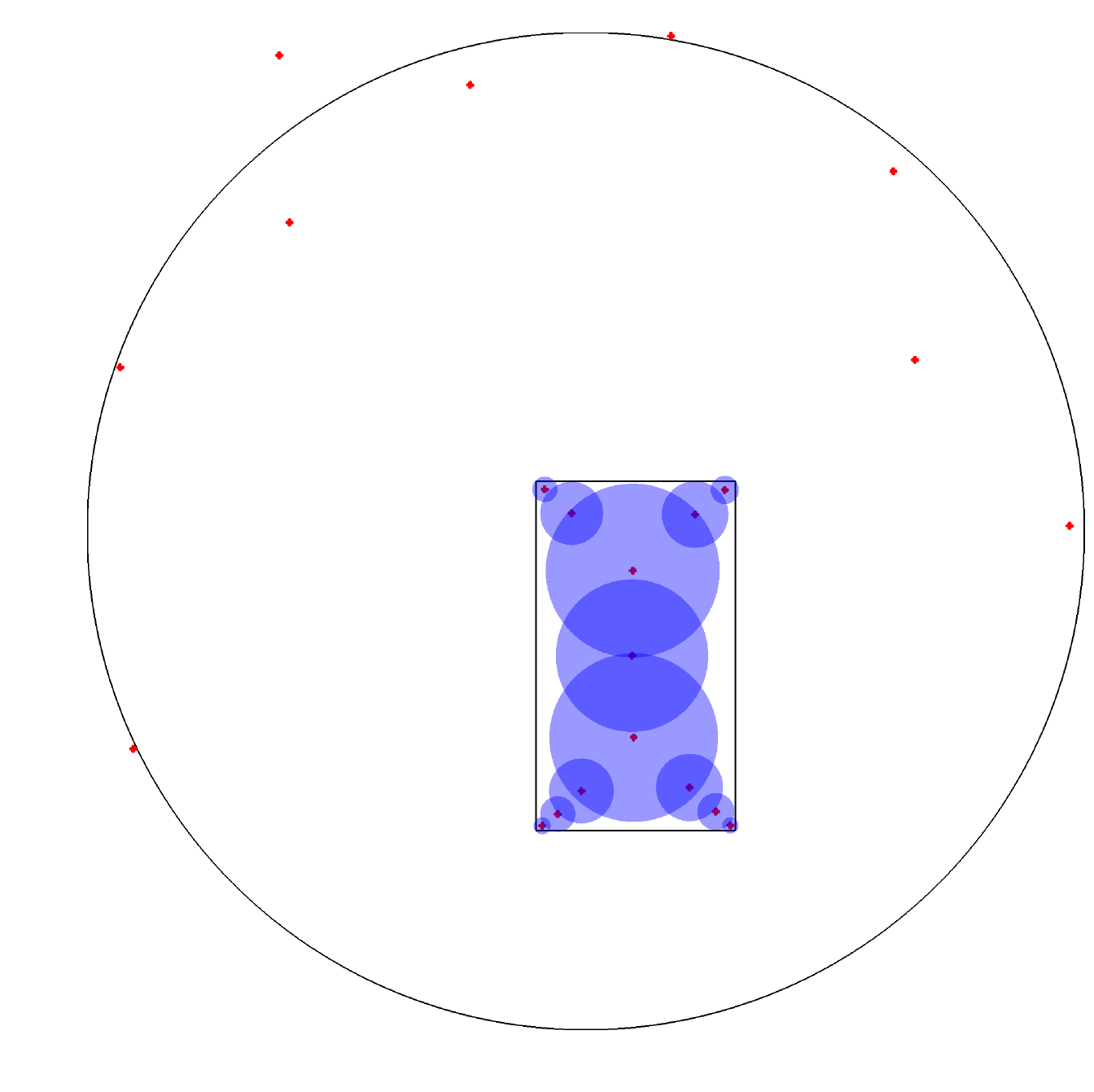}}
\caption{Example of reconstruction of a rectangle cavity for different values of $n$, the number of atoms (the red dots represent the centers of the disks). Note that some disks are degenerate (i.e. have zero radii).  \label{fig:rectangle} }
\end{figure}

\begin{figure}[h]
\centering
	\subfigure[$n=1$]{\includegraphics[width=0.4\textwidth]{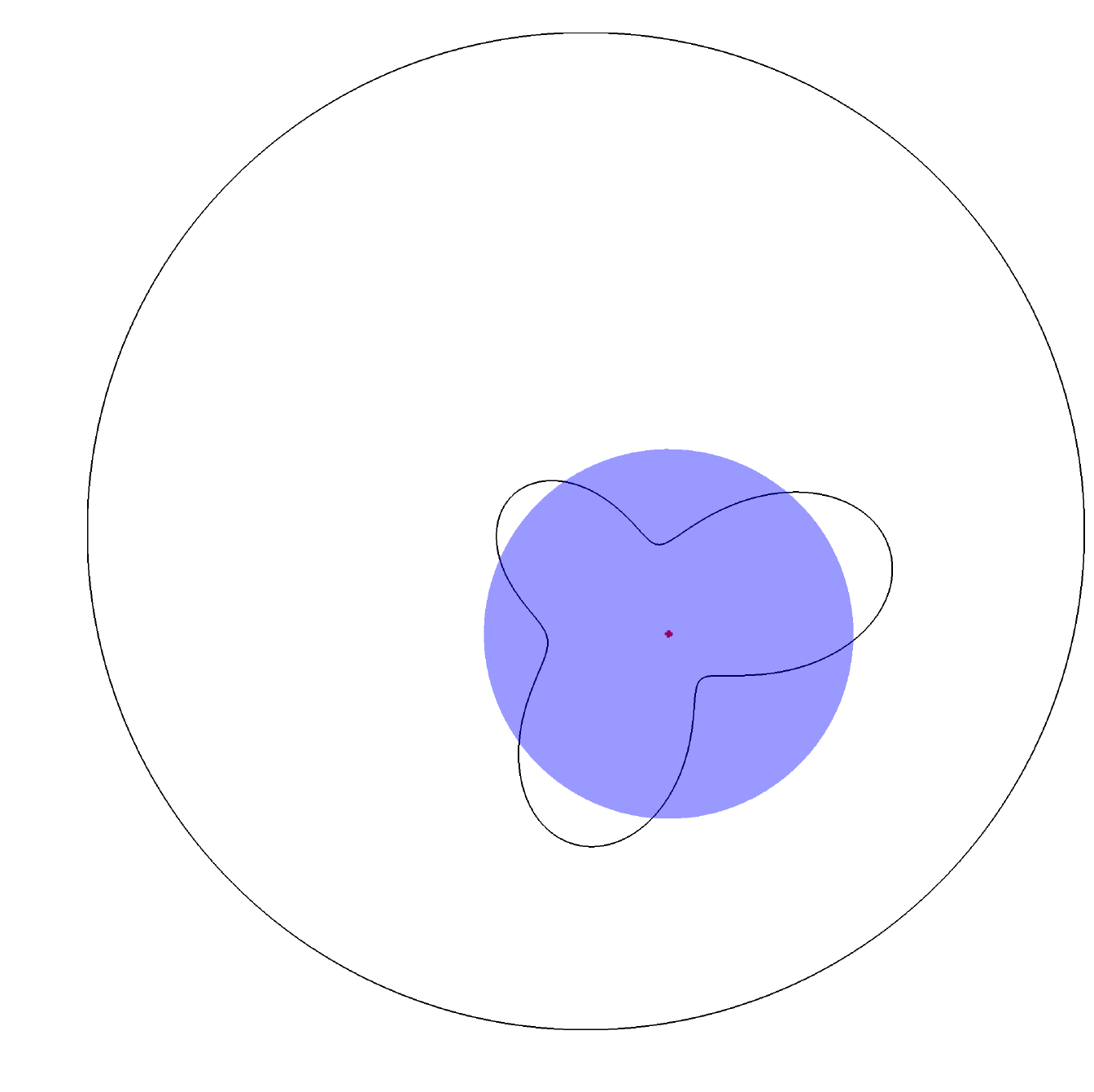}}\hspace{1cm}
	\subfigure[$n=3$]{\includegraphics[width=0.4\textwidth]{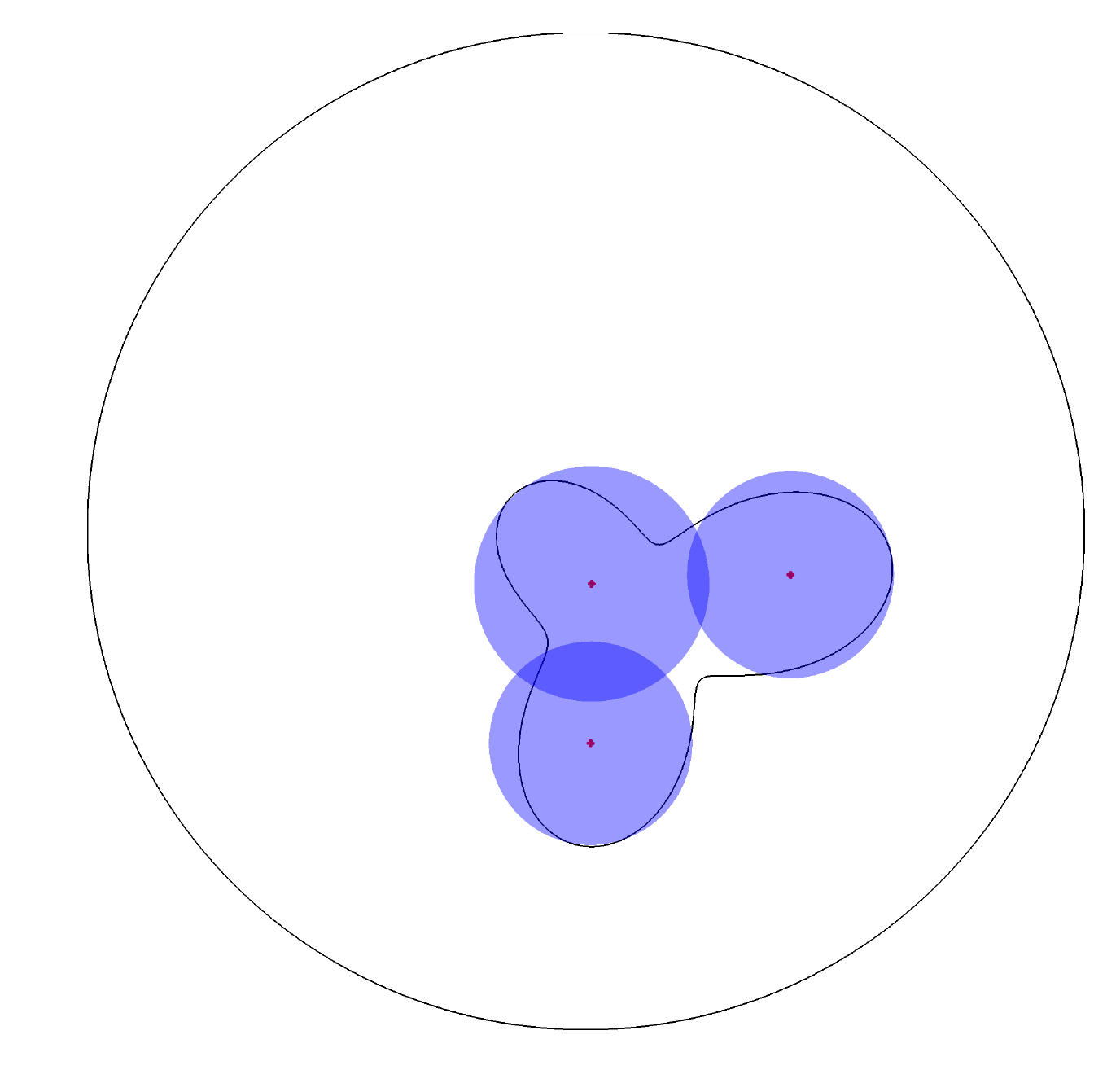}}\\
 	\subfigure[$n=8$]{\includegraphics[width=0.4\textwidth]{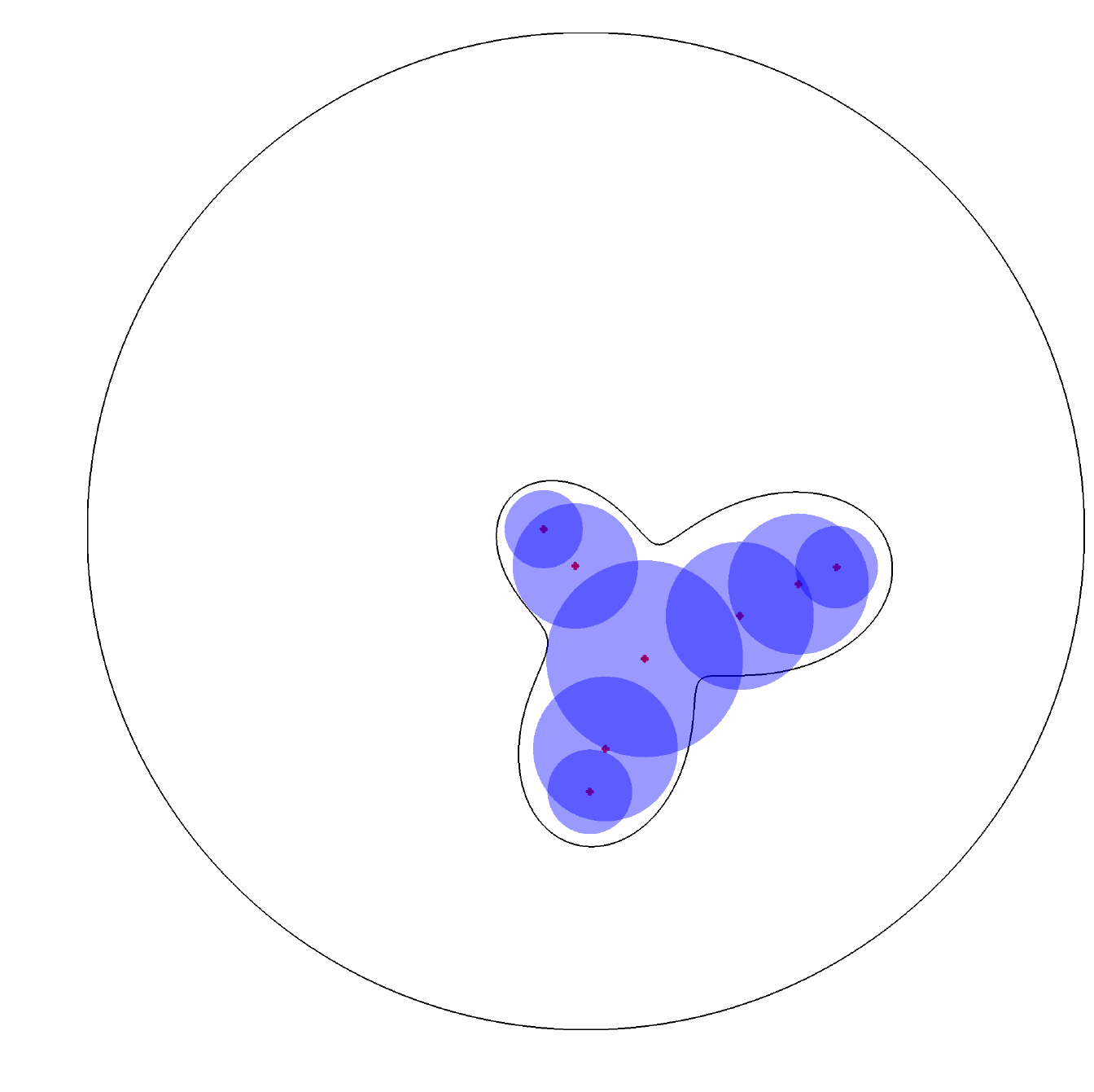}}\hspace{1cm}
	\subfigure[$n=17$]{\includegraphics[width=0.4\textwidth]{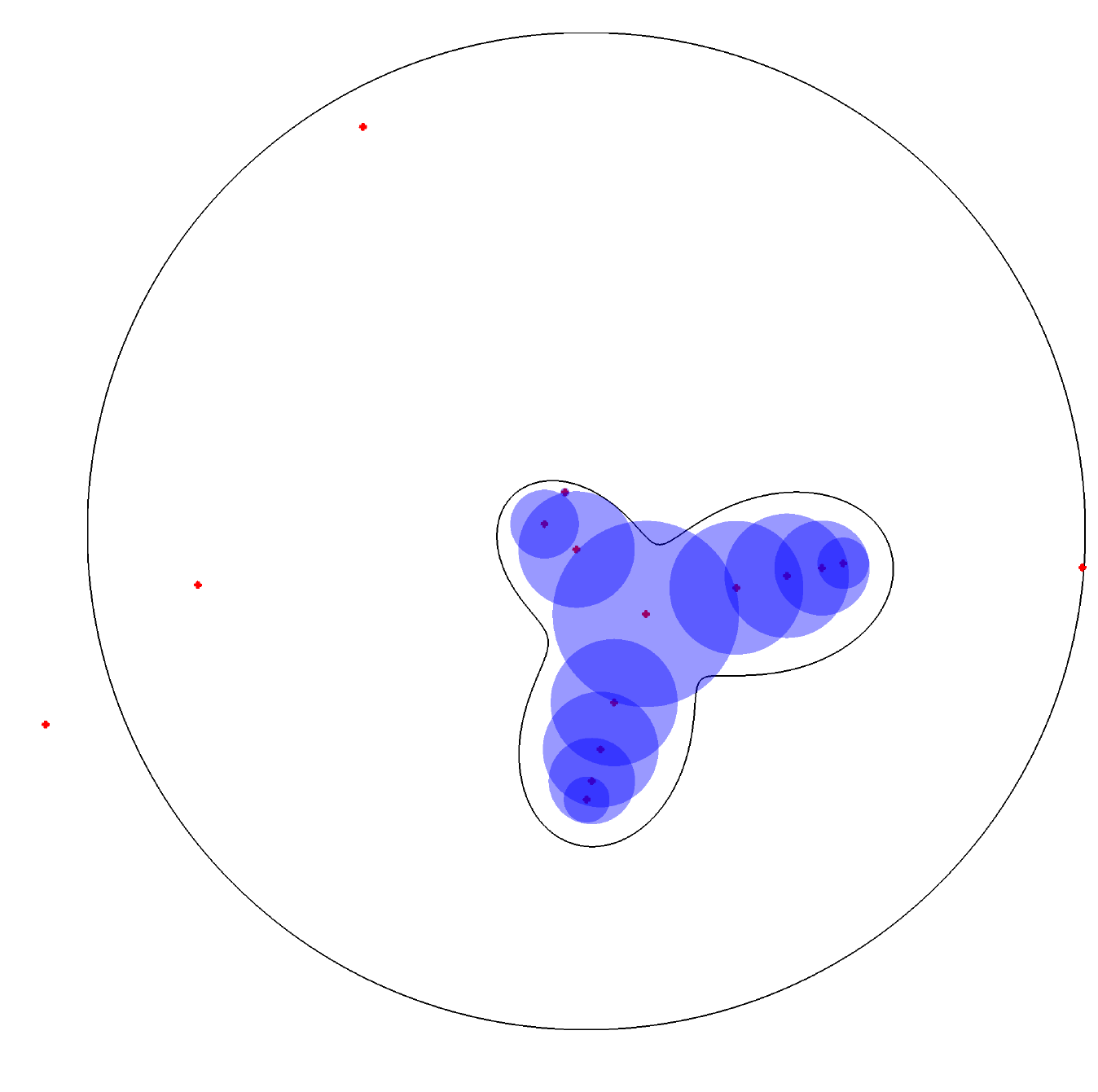}}
\caption{Example of reconstruction of a clover shaped cavity  for different values of $n$.  \label{fig:trefle} }
\end{figure}

\begin{figure}[h]
\centering
	\subfigure[$n=1$]{\includegraphics[width=0.4\textwidth]{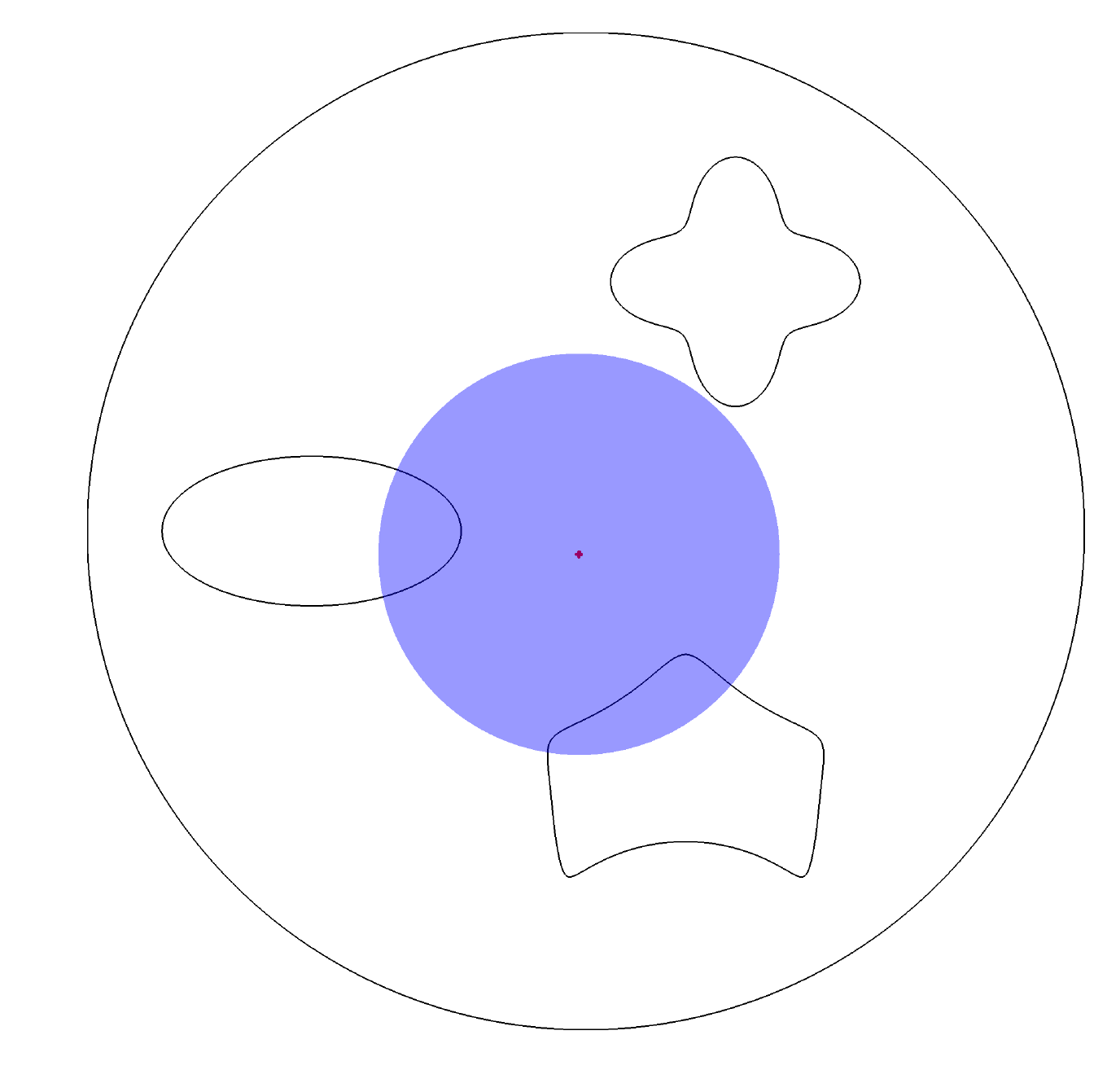}}\hspace{1cm}
	\subfigure[$n=3$]{\includegraphics[width=0.4\textwidth]{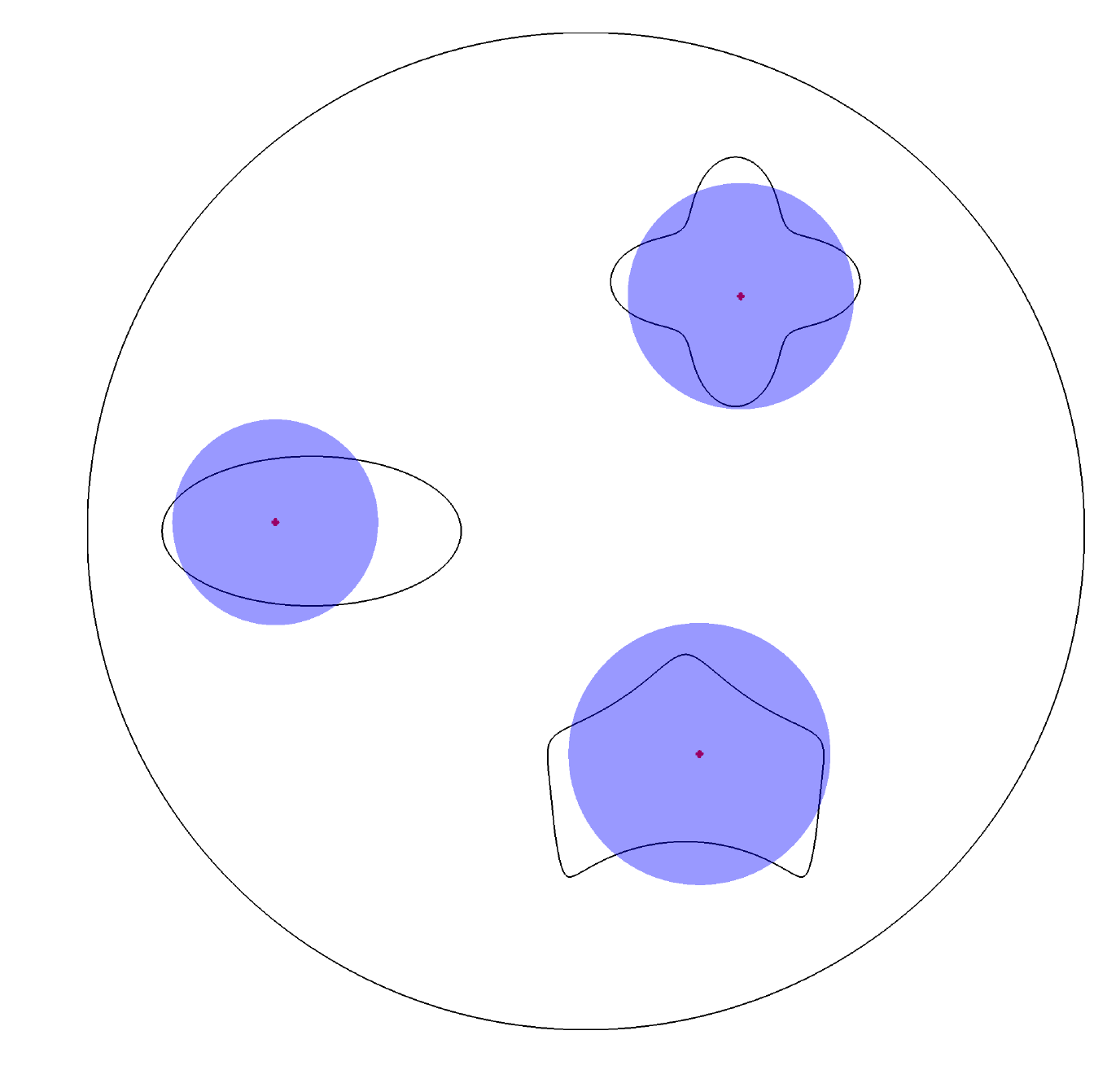}}\\
 	\subfigure[$n=9$]{\includegraphics[width=0.4\textwidth]{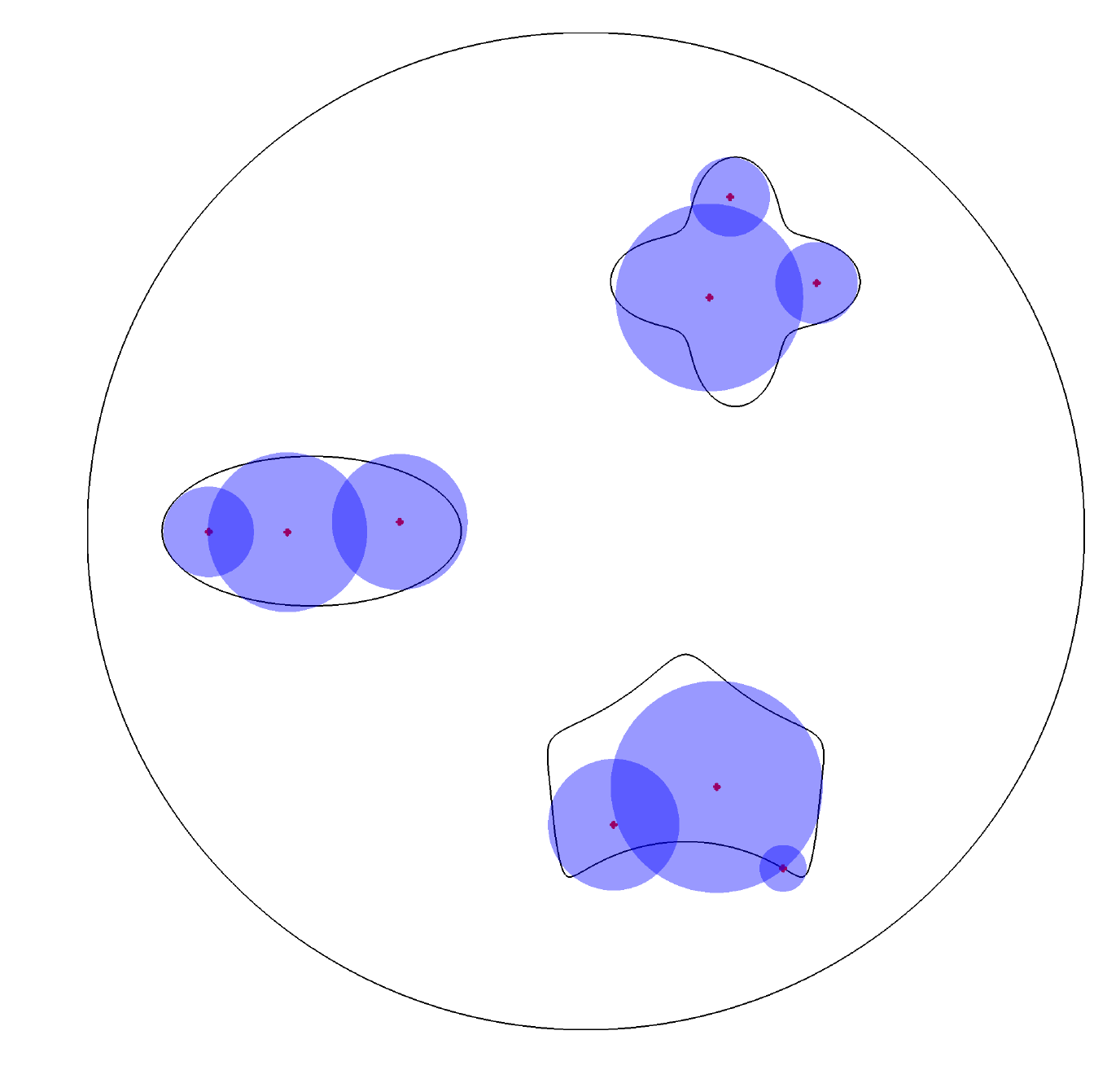}}\hspace{1cm}
	\subfigure[$n=21$]{\includegraphics[width=0.4\textwidth]{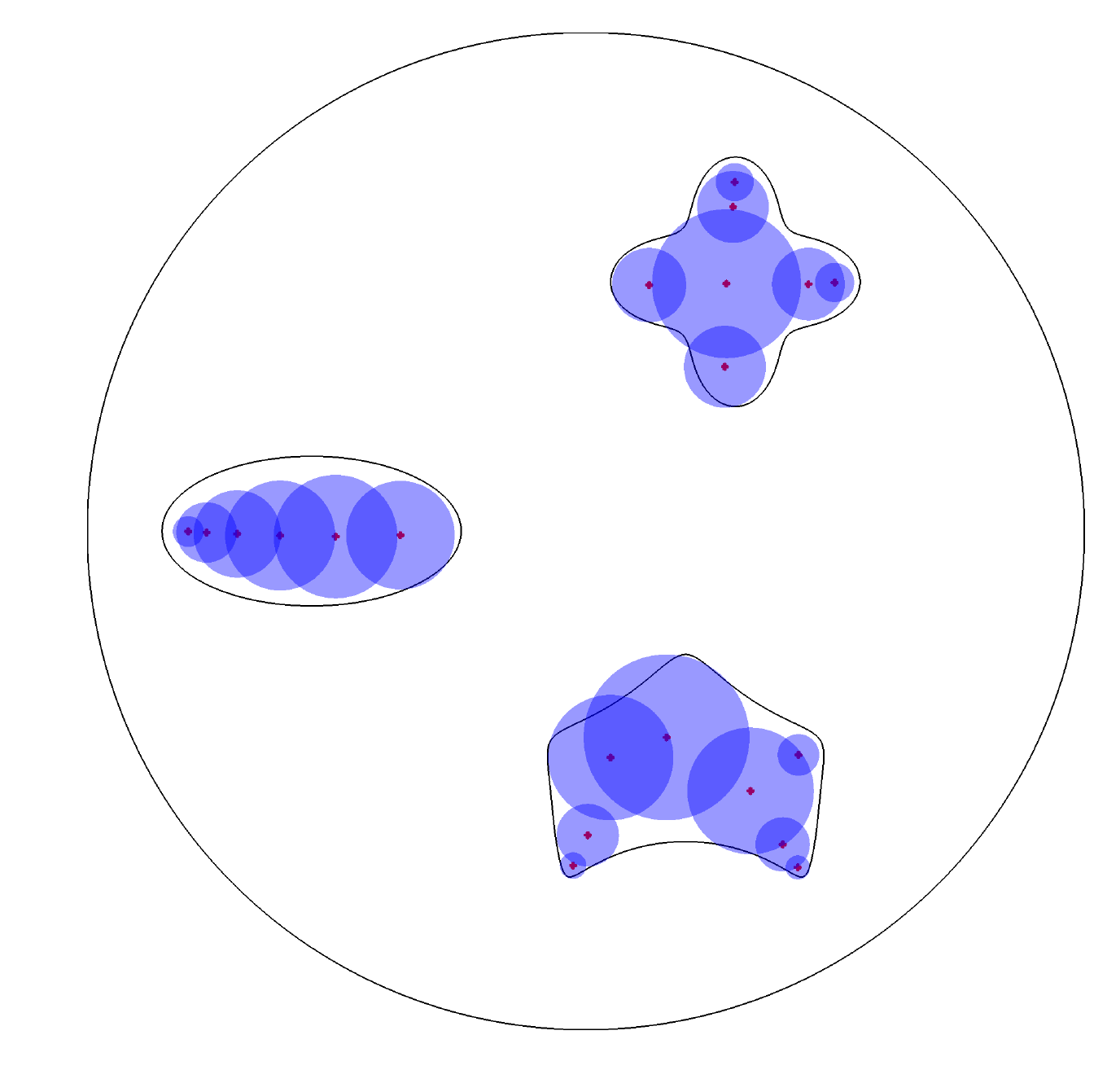}}
\caption{Example of reconstruction of a multiply connected cavity for different values of $n$.  \label{fig:multi} }
\end{figure}

\begin{figure}[h]
\centering
	\subfigure[$n=16$\label{fig:pb:16}]{\includegraphics[width=0.4\textwidth]{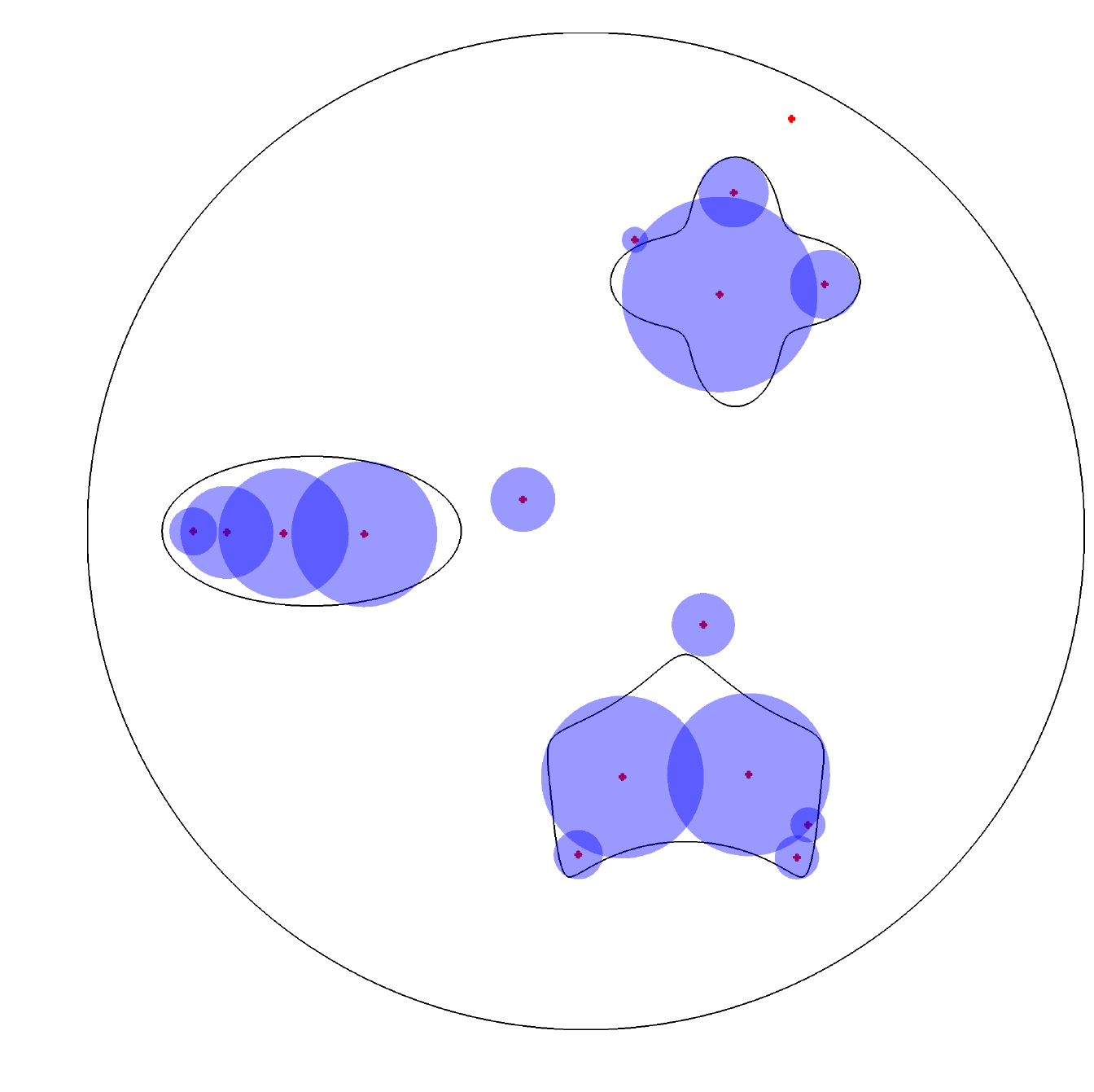}}\hspace{1cm}
	\subfigure[$n=17$\label{fig:pb:17}]{\includegraphics[width=0.4\textwidth]{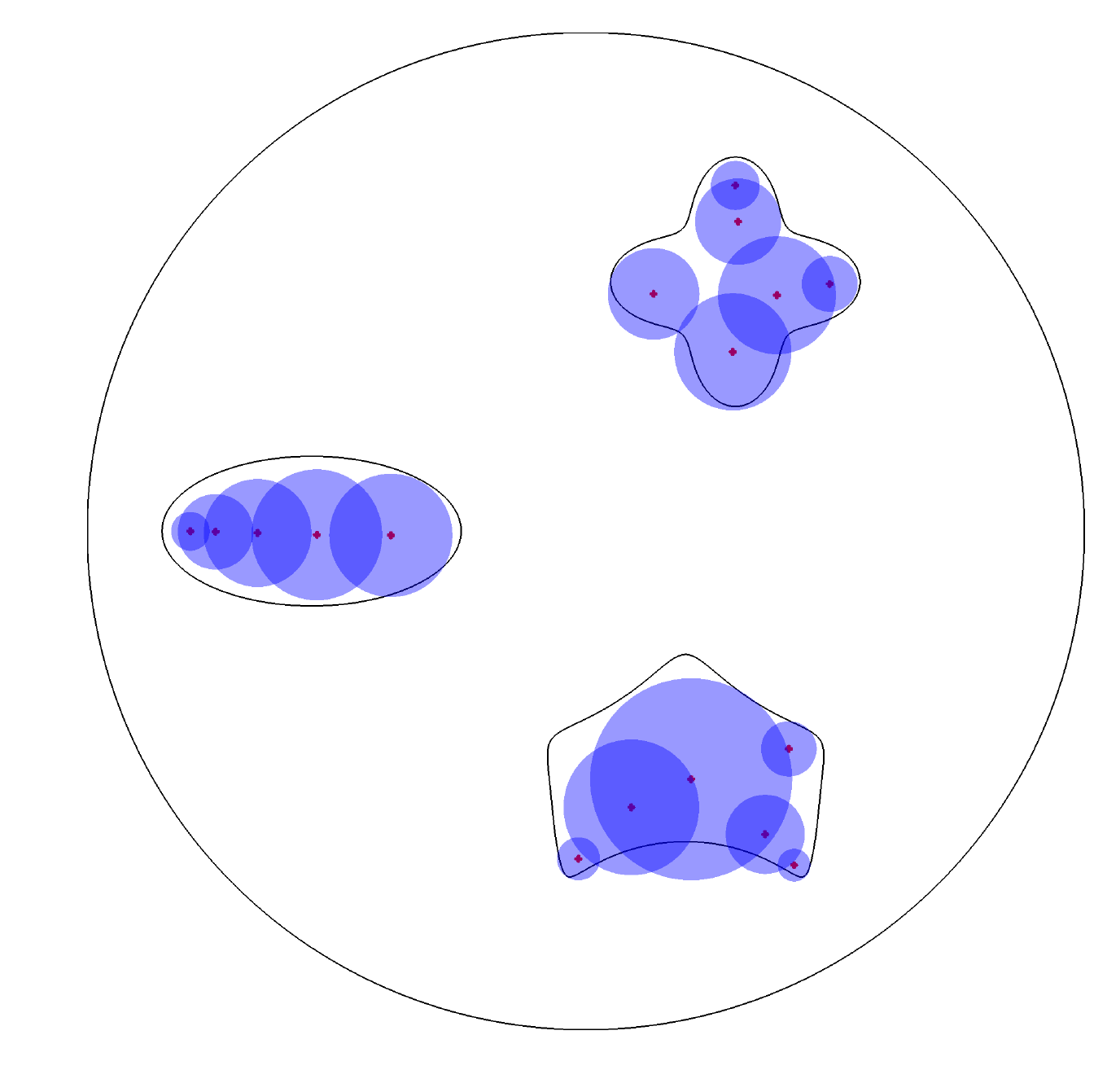}}
\caption{Sensitivity of the reconstruction with respect to $n$: appearance and disappearance of spurious disks outside the cavity.  \label{fig:pb} }
\end{figure}

\appendix
\section{Appendix}\label{sect:appendix}

The next lemma generalizes to the case of a multiply connected boundary the result given in \cite[Lemma 8.14]{McL00} for a simply connected boundary. Our proof is slightly different from the one given there, although it also uses the Fredholm alternative. The notation are those of Section \ref{sect:EI} and we recall that the assumption ${\rm Cap}({\gamma})\neq 1$ is supposed to hold true.
\begin{lemma}
\label{lem:EImult}
Let $f\in \Hcp{\gamma}$ and $\mathbf b=(b_{1}, \dots, b_{N})^{\mathbf T}\in \mathbb R^N$. Then, there exists a unique density $\hat q=(\hat q_{1},\dots,\hat q_{N})\in  \Hcm{\gamma}$ and a unique $\mathbf c=(c_{1}, \dots, c_{N})^{\mathbf T}\in \mathbb R^N$ satisfying the system of equations
 system of equations
\begin{subequations}
\label{eq:EImultiple}
\begin{alignat}{3}
{\mathsf S}_{\gamma}\hat q + \mathbf c &=f,&\quad&\text{on }\gamma\\
\langle \hat q_{k} , 1\ranglemp{{\gamma_{k}}} &=b_k,&\quad&  k=1, \dots,N.
\end{alignat}
\end{subequations}
\end{lemma}
\begin{proof}
Let us introduce the  operator $\mathcal A$ defined on $\Hcm{\gamma}\times \mathbb R^N$ by 
$$
\mathcal A= 
\begin{pmatrix}
{\mathsf S}_{\gamma}& \text{Id}_{\mathbb R^N} \\
{\mathsf J}_{\gamma}  & 0
\end{pmatrix},
$$
in which
$$
{\mathsf J}_{\gamma}:= \begin{pmatrix} \langle  \cdot, 1\ranglemp{{\gamma_{1}}}  &&0\\  &\ddots& \\ 0 & &\langle \cdot, 1\ranglemp{{\gamma_{N}}}  \end{pmatrix}.
$$
Using this notation, system \eqref{eq:EImultiple} simply reads 
$$
\mathcal A \begin{pmatrix}\hat q \\ \mathbf c \end{pmatrix}=  \begin{pmatrix}f \\ \mathbf b \end{pmatrix}.
$$
Clearly, $\mathcal A$ defines a bounded operator from $\Hcm{\gamma}\times \mathbb R^N$ onto $\Hcp{\gamma}\times \mathbb R^N$. Moreover, from the decomposition
$$
\mathcal A =\mathcal A_{0}+\mathcal C,
$$
with
$$
\mathcal A_{0}= 
\begin{pmatrix}
{\mathsf S}_{\gamma}& 0 \\
0  &  \text{Id}_{\mathbb R^N}
\end{pmatrix},\qquad
\mathcal C= 
\begin{pmatrix}
0& \text{Id}_{\mathbb R^N} \\
{\mathsf J}_{\gamma}  &  -\text{Id}_{\mathbb R^N}
\end{pmatrix},
$$
it is clear that $\mathcal A$ is a Fredholm operator ($\mathcal C$ is a finite rank operator and ${\mathsf S}_{\gamma}$ is boundedly invertible). 

Hence, \eqref{eq:EImultiple} is uniquely solvable if and only if the corresponding homogeneous problem admits only the null solution. Let then $(\hat q, \mathbf c)\in  \Hcm{\gamma}\times {\mathbb R}^N$ satisfying $\mathcal A \begin{pmatrix}\hat q \\ \mathbf c \end{pmatrix}=0$ and denote by $u:=\mathscr S_{\gamma}\hat q$ the single layer potential  corresponding corresponding to $\hat q$. Then, according to the first equation in \eqref{eq:EImultiple} (with $f=0$), $u$ is constant on $\gamma$ and thus constant in $\omega$. On the other hand, the second equation in \eqref{eq:EImultiple} implies that $\hat q\in \widehat  {\mathcal H} (\gamma)$ and hence $u$ vanishes outside $\omega$. Consequently, $\hat q_{k} = -\left[ \partial_{n}u \right]_{\gamma_{k}}=0$ and $\mathbf c=-{\mathsf S}_{\gamma}\hat q=0$. 
\end{proof}

\bibliographystyle{amsplain}

\end{document}